\documentclass[twoside,leqno,nosumlimits]{article}

\usepackage[letterpaper]{geometry}

\usepackage{siamproceedings}

\usepackage[T1]{fontenc}
\usepackage{amsfonts}
\usepackage{graphicx}
\usepackage{epstopdf}
\usepackage{enumitem}
\ifpdf
  \DeclareGraphicsExtensions{.eps,.pdf,.png,.jpg}
\else
  \DeclareGraphicsExtensions{.eps}
\fi

\newsiamremark{remark}{Remark}
\newsiamremark{hypothesis}{Hypothesis}
\crefname{hypothesis}{Hypothesis}{Hypotheses}
\newsiamthm{claim}{Claim}

\usepackage{amsopn}

\clearpage{}%

\usepackage[usenames,dvipsnames]{xcolor}
\usepackage{fancyhdr}

\usepackage{amsmath,amsfonts,amsbsy,amsgen,amscd,mathrsfs,amssymb}
\usepackage{stackengine}

\graphicspath{{figures/}}

\usepackage{tikz}
\usetikzlibrary{arrows,chains,matrix,positioning,scopes}

\usepackage{xcolor} %

\definecolor{cit}{rgb}{0.91,0.39,0.16}	%

\definecolor{dark-gray}{gray}{0.3}

\definecolor{dkgray}{rgb}{.3,.3,.3}
\definecolor{medgray}{rgb}{.5,.5,.5}
\definecolor{ltgray}{rgb}{.7,.7,.7}
\definecolor{dkblue}{rgb}{0,0,.5}
\definecolor{medblue}{rgb}{0,0,.75}
\definecolor{ltblue}{rgb}{0.97,0.97,1}
\definecolor{rust}{rgb}{0.5,0.1,0.1}
\definecolor{ltyellow}{rgb}{1, 1, 0.9}

\usepackage{soul}
\sethlcolor{ltyellow}

\usepackage[normalem]{ulem}

\usepackage{bm} %

\usepackage[utf8]{inputenc} %

\usepackage[english]{babel} %

\setlist{noitemsep} %
\theoremstyle{definition}
\newtheorem{example}[theorem]{Example}

\RequirePackage[framemethod=default]{mdframed}

\newmdenv[skipabove=6pt,
skipbelow=6pt,
rightline=false,
leftline=true,
topline=false,
bottomline=false,
backgroundcolor=ltyellow,
linecolor=cit,
innerleftmargin=10pt,
innerrightmargin=10pt,
innertopmargin=0pt,
innerbottommargin=5pt,
leftmargin=0cm,
rightmargin=0cm,
linewidth=4pt]{iBox}

\newmdenv[skipabove=0pt,
skipbelow=0pt,
backgroundcolor=ltblue,
linecolor=dkblue,
linewidth=2pt,
rightline=false,
leftline=false,
topline=false,
bottomline=false,
innerleftmargin=7pt,
innerrightmargin=10pt,
innertopmargin=6pt,
innerbottommargin=6pt,
leftmargin=0cm,
rightmargin=0cm,
innerbottommargin=5pt]{aBox}

\newmdenv[skipabove=10pt,
skipbelow=10pt,
backgroundcolor=white,
linecolor=dkblue,
linewidth=0.5pt,
rightline=true,
leftline=true,
topline=true,
bottomline=true,
innerleftmargin=10pt,
innerrightmargin=0.5in,
innertopmargin=5pt,
innerbottommargin=5pt,
leftmargin=0cm,
rightmargin=0cm]{lfBox}

\usepackage{algorithmicx}
\usepackage[noend]{style/myalgpseudocode}

\algrenewcommand\alglinenumber[1]{\sf\scriptsize\color{dkgray}{#1}}
\algrenewcommand\algorithmicrequire{\textbf{Input:}}
\algrenewcommand\algorithmicensure{\textbf{Output:}}

\algdef{SE}[SUBALG]{Indent}{EndIndent}{}{\algorithmicend\ }
\algtext*{Indent}
\algtext*{EndIndent}

\newcommand{\econst}{\mathrm{e}}
\newcommand{\iunit}{\mathrm{i}}

\newcommand{\eps}{\varepsilon}

\renewcommand{\phi}{\varphi}

\newcommand{\vct}[1]{\bm{#1}}
\newcommand{\mtx}[1]{\bm{#1}}
\newcommand{\set}[1]{\mathsf{#1}}
\newcommand{\coll}[1]{\mathcal{#1}}

\newcommand{\term}[1]{\textit{#1}}

\newcommand{\N}{\mathbb{N}}

\newcommand{\R}{\mathbb{R}}

\newcommand{\C}{\mathbb{C}}
\newcommand{\F}{\mathbb{F}}

\newcommand{\M}{\mathbb{M}}
\newcommand{\Sym}{\mathbb{H}}

\renewcommand{\Re}{\operatorname{Re}}

\newcommand{\range}{\operatorname{range}}
\newcommand{\nullsp}{\operatorname{null}}

\newcommand{\lspan}{\operatorname{span}}

\newcommand{\rank}{\operatorname{rank}}
\newcommand{\trace}{\operatorname{Tr}}
\newcommand{\diag}{\operatorname{diag}}
\newcommand{\Id}{\mathbf{I}}

\newcommand{\psdle}{\preccurlyeq}

\newcommand{\abs}[1]{\vert {#1} \vert}
\newcommand{\norm}[1]{\Vert {#1} \Vert}
\newcommand{\ip}[2]{\langle {#1}, \ {#2} \rangle}

\newcommand{\lnorm}[1]{\left\Vert {#1} \right\Vert}

\newcommand{\grad}{\nabla}

\newcommand{\Expect}{\operatorname{\mathbb{E}}}
\newcommand{\Var}{\operatorname{Var}}

\newcommand{\Probe}{\mathbb{P}}
\newcommand{\Prob}[1]{\Probe\left\{ #1 \right\}}

\newcommand{\condbar}{\, \vert \,}

\newcommand{\normal}{\textsc{normal}}

\DeclareFontFamily{U}{matha}{\hyphenchar\font45}
\DeclareFontShape{U}{matha}{m}{n}{
  <-6> matha5 <6-7> matha6 <7-8> matha7
  <8-9> matha8 <9-10> matha9
  <10-12> matha10 <12-> matha12
  }{}

\DeclareSymbolFont{matha}{U}{matha}{m}{n}
\DeclareMathSymbol{\abscont}{3}{matha}{"CE}

\usepackage{mathtools}

\makeatletter
\def\paragraph{\@startsection{paragraph}{4}%
  \z@\z@{-\fontdimen2\font}%
  {\normalfont\scshape}}
\makeatother

\clearpage{}%

\begin{document}

\newcommand\relatedversion{}

\title{\Large Applied Random Matrix Theory\relatedversion}
    \author{Joel A.~Tropp\thanks{Department of Computing and Mathematical Sciences, Caltech, Pasadena, CA (\email{jtropp@caltech.edu}, \url{https://tropp.caltech.edu/}).}}

\date{1 October 2025}

\maketitle

\fancyfoot[R]{\scriptsize{To appear, \textsl{Proceedings of the 2026 International Congress of Mathematicians}}}

\begin{abstract} %
Random matrices now play a role in many parts of computational mathematics.
To advance these applications, it is desirable to have tools that are flexible, easy to use, and powerful.
Over the last 25 years, researchers have developed a remarkable family of results,
called matrix concentration inequalities, that meet the criteria.
This paper offers an invitation to the field of matrix concentration and its multifarious applications.
\end{abstract}

\section{Motivation}

Random matrix theory emerged from applications in \textbf{statistics} (sample covariance matrices~\cite{Wis28:Generalised-Product})
and in \textbf{nuclear physics} (Hamiltonians of heavy nuclei~\cite{Wig55:Characteristic-Vectors}).
The subject established itself within the mathematical firmament through deep connections
to \textbf{number theory} (zeros of the Riemann zeta function~\cite{Mon73:Pair-Correlation}),
\textbf{operator algebras} (free product factors~\cite{VDN92:Free-Random}),
\textbf{combinatorics} (longest increasing subsequence~\cite{BDJ99:Distribution-Length}),
and beyond.
Over the last generation, random matrices have attracted new attention in computational mathematics,
starting with work in \textbf{algorithms}~\cite{LLR95:Geometry-Graphs} and in \textbf{quantum information}~\cite{AW02:Strong-Converse},
and expanding in waves that have touched the farthest shores.

The classic literature emphasizes random matrices that enjoy a %
harmonious mathematical structure,
such as the independent and identically distributed (iid) entries of a Wigner matrix~\cite{Wig55:Characteristic-Vectors}.
By now, we understand these models comprehensively~\cite{BS10:Spectral-Analysis,PS10:Eigenvalue-Distribution}.
In contrast, contemporary applications often lead to random matrices of more baroque construction,
such as a random set of columns drawn from a fixed data matrix~\cite[Sec.~5.2]{Tro15:Introduction-Matrix}.
The standard methods for studying random matrices have relatively little to say about these strange examples.

To advance into this frontier, researchers have developed new tools for applied random matrix theory,
collectively called \term{matrix concentration inequalities}~\cite{Tro15:Introduction-Matrix}.
These results offer several key benefits:

\begin{enumerate}
\item	\textbf{Flexibility.} They apply to a large family of random matrices.
\item	\textbf{Ease of use.}  They reduce the analysis to a calculation of simple summary statistics.
\item	\textbf{Power.} They accurately describe features of the random matrix model.
\end{enumerate}

\noindent
Matrix concentration has roots in Banach space geometry~\cite{TJ74:Moduli-Smoothness,LP86:Inegalites-Khintchine,Rud99:Random-Vectors};
it also owes a heavy debt to research on quantum information theory~\cite{AW02:Strong-Converse}.
The core results crystallized in 2010 through the efforts of
Oliveira~\cite{Oli10:Concentration-Adjacency} and the author~\cite{Tro11:Freedmans-Inequality,Tro12:User-Friendly}.
My monograph~\cite{Tro15:Introduction-Matrix} collected many applications and popularized the theory.
Matrix concentration soon became textbook material~\cite{Ver18:High-Dimensional-Probability,Wai19:High-Dimensional-Statistics},
and it has now informed thousands of research papers.

The purpose of this memoir is to introduce the fundamental matrix concentration
inequalities for independent sums and martingales.  I will illustrate these tools through
eight contemporary applications in uncertainty quantification, numerical linear algebra, spectral graph theory,
high-dimensional statistics, and quantum information science.
Parts of this work are adapted from my monograph~\cite{Tro15:Introduction-Matrix}
and lecture notes~\cite{Tro19:Matrix-Concentration-LN}.

\subsection{Concentration of random matrices}

A \term{random matrix} is a matrix whose entries
are random variables, not necessarily independent from each other.
The distinctive concern of random matrix theory is the action
of the random matrix as a random linear map between linear spaces.
In that vein, we can investigate its geometric properties
(reflected in operator norms or its action on sets)
and its spectral features (singular values and singular vectors,
eigenvalues and eigenvectors in the square case).

Matrix concentration studies the deviation
of a random matrix $\mtx{S} \in \C^{d_1 \times d_2}$ from its expectation $\Expect[\mtx{S}] \in \C^{d_1 \times d_2}$,
as measured in the spectral norm $\norm{\cdot}$.
For levels $t \geq 0$, we would like to control the tail probability
\begin{equation} \label{eqn:mtx-conc-nominal}
\Prob{ \norm{ \mtx{S} - \Expect[ \mtx{S} ] } \geq t }
	\quad \leq \quad \underline{\hspace{2pc} \textbf{???} \hspace{2pc}}.
\end{equation}
The expectation of the random matrix is computed componentwise;
we always assume that this expectation is defined and finite.
The spectral norm is also called the $\ell_2$
operator norm;  it coincides with the largest singular value.
If we control the deviation $\norm{\mtx{S} - \Expect[\mtx{S}]}$,
we also control other characteristics of the random matrix:

\begin{enumerate}
\item	\textbf{Singular values.}  The singular values of $\mtx{S}$ and $\Expect[\mtx{S}]$ are close.

\item	\textbf{Singular vectors.}  The singular vectors of $\mtx{S}$ and $\Expect[\mtx{S}]$ are close for isolated singular values.

\item	\textbf{Linear functionals.}  All linear functionals of $\mtx{S}$ and $\Expect[\mtx{S}]$ are comparable.

\end{enumerate}

\noindent
Many practical problems reduce to one of these considerations.

\subsection{Random matrix models}

We cannot hope to make progress on the question~\eqref{eqn:mtx-conc-nominal}
without imposing some restrictions on the random matrix.
The challenge is to identify models that capture a wide range of examples,
yet offer enough scaffolding to support strong theoretical guarantees.
This paper focuses on two models inspired by
the most basic scalar stochastic processes: independent sums and martingales.
Our attention to these templates will be rewarded by a host of applications.
\Cref{sec:beyond} outlines recent research and alternative models.

\subsubsection{Independent sums}

A random matrix $\mtx{S} \in \C^{d_1 \times d_2}$ follows the \term{independent sum model}
when it admits the decomposition
\[
\mtx{S} = \sum_{k=1}^n \mtx{X}_k
\quad\text{where the family $(\mtx{X}_1, \dots, \mtx{X}_n)$ is statistically independent.}
\]
We use the term ``statistical independence'' to mark a contrast against linear independence.
Let us emphasize that the entries within any particular matrix $\mtx{X}_k$ may exhibit dependencies,
but $\mtx{X}_k$ cannot provide any information about events involving $(\mtx{X}_j : j \neq k)$.
Moreover, there may be many distinct decompositions of a random matrix as an independent sum.
\emph{Inter alia}, the independent sum model describes a random Monte Carlo approximation
of a fixed matrix as an average of simple unbiased estimates~\cite[Sec.~6.2]{Tro15:Introduction-Matrix}. %

We would like to capture information about the concentration of the independent sum $\mtx{S}$
through summary statistics. %
In the scalar setting, we can achieve this goal with the Bernstein inequality~\cite[Thm.~2.10]{BLM13:Concentration-Inequalities},
which is arguably the most useful probability inequality for an independent sum of scalars.
The matrix Bernstein inequality~\cite[Thm.~6.2]{Tro12:User-Friendly} offers a perfect
analogue in the matrix setting, where it may be the single most productive tool for
studying random matrices. 
We present this result as \cref{thm:matrix-bernstein}.
\Cref{sec:bernstein-appl} showcases applications to active subspace methods,
stochastic rounding, graph sparsification, and quantum state tomography.

\subsubsection{Martingales}

A \term{matrix martingale}
is a sequence $(\mtx{S}_0, \mtx{S}_1, \mtx{S}_2, \dots)$
of random matrices with common dimension $d_1 \times d_2$
that satisfies the ``status quo'' property:
\[
\Expect[ \mtx{S}_{k+1} \condbar \mtx{S}_0, \dots, \mtx{S}_{k} ] = \mtx{S}_{k}
	\quad\text{for each $k = 0,1,2, \dots$.}
\]
Given what we know so far, our best estimate for the
next matrix $\mtx{S}_{k+1}$ is the current matrix $\mtx{S}_k$.
As a simple example, the partial sums of an independent
family of centered random matrices compose a matrix martingale.
More generally, matrix martingales can model complicated sequences
of random matrices that are revealed one step at a time,
such as the iterates of a randomized algorithm that performs
a linear algebra computation.

One of the most powerful tools for studying scalar martingale sequences
is the Freedman inequality~\cite{Fre75:Tail-Probabilities},
which is the martingale analog of the Bernstein inequality.
The matrix Freedman inequality~\cite{Oli10:Concentration-Adjacency,Tro11:Freedmans-Inequality}
generalizes the scalar result to matrix martingales.
It allows us to treat sequences of random matrices that evolve adaptively,
and it provides uniform control on the whole trajectory.
We present this result as \cref{thm:matrix-freedman}.
\Cref{sec:freedman-appl} highlights applications
to online covariance estimation,
Cholesky decomposition with stochastic rounding,
fast graph Laplacian solvers,
and randomized approximation of quantum Hamiltonians.

\subsection{Notation}

The Pascal notation $\coloneqq$ or $\eqqcolon$ generates a definition.
The symbols $\vee$ and $\wedge$ denote the infix maximum and minimum.
We work over a scalar field $\F \in \{\R, \C\}$, which is real or complex.
The star ${}^*$ delivers the (conjugate) transpose of a vector or matrix.
The linear space $\F^d$ is equipped with the standard inner product $\ip{\vct{b}}{\vct{a}} \coloneqq \vct{b}^* \vct{a}$
and the associated $\ell_2$ norm $\norm{\cdot}$.

In the space $\F^d$, the standard basis elements are $\mathbf{e}_1, \dots, \mathbf{e}_d$,
and $\mathbf{1}$ is the vector of ones.
The standard basis elements in a matrix space, such as $\F^{d_1 \times d_2}$,
are $\mathbf{E}_{ij}$.  The identity matrix is $\Id$.
Dimensions are determined by context.
We write $a_{ij}$ for the entries of a matrix $\mtx{A}$,
and we use the colon operator $\mtx{A}(i, :)$ or $\mtx{A}(:, j)$
to extract the $i$th row or $j$th column.
For matrices, $\norm{\cdot}$ is the spectral norm (aka the Schatten $\infty$-norm),
$\norm{\cdot}_{\rm F}$ denotes the Frobenius norm (aka the Schatten $2$-norm), and
$\norm{\cdot}_{1}$ refers to the trace norm (aka the Schatten $1$-norm). 

The linear space $\M_d(\F)$ consists of $d \times d$ matrices with entries in $\F$,
on which the operator $\trace$ computes the trace.
The real-linear subspace $\Sym_{d}(\F)$ consists of the $d \times d$ Hermitian (i.e., conjugate-symmetric) matrices.
The cone $\Sym_{d}^+(\F)$ contains the $d \times d$ positive-semidefinite (psd) matrices,
and $\Sym_d^{++}(\F)$ is the cone of positive-definite matrices.
For Hermitian matrices, the semidefinite partial order $\mtx{A} \psdle \mtx{H}$
signifies that $\mtx{H} - \mtx{A}$ is psd, while the maps $\lambda_{\max}$ and $\lambda_{\min}$
return the (algebraic) maximum and minimum eigenvalues.

The operator $\Expect$ computes the expectation of a random variable,
and $\Probe$ reports the probability of an event.
The symbol $\sim$ means ``has the distribution.''
For a real random variable $X$, the function $\sup X$
specifies its least upper bound on the probability space.
Within a master probability space, the function $\sigma(\cdot)$ returns the sigma-algebra generated
by its arguments; the empty argument returns the trivial sigma-algebra.

The big-$\coll{O}$ notation is interpreted per the conventions
of computer science.

\section{The matrix Bernstein inequality}

Among matrix concentration inequalities, the single most important
result is the matrix extension of the scalar Bernstein inequality~\cite[Thm.~2.10]{BLM13:Concentration-Inequalities}.
The matrix Bernstein inequality was established independently by
Roberto I.~Oliveira~\cite{Oli10:Concentration-Adjacency} in late 2009
and the author~\cite{Tro12:User-Friendly} in early 2010.
We state the version of this result from~\cite[Thm.~6.1.1]{Tro15:Introduction-Matrix}.

\begin{theorem}[Matrix Bernstein] \label{thm:matrix-bernstein}
Consider a statistically independent family $(\mtx{X}_1, \dots, \mtx{X}_n)$
of $d_1 \times d_2$ random matrices, real or complex, %
that are centered and uniformly bounded: $\Expect[ \mtx{X}_k ] = \mtx{0}$ and $\norm{ \mtx{X}_k } \leq B$ for each index $k$.
Form the sum, and calculate its matrix variance statistic:
\begin{equation} \label{eqn:matrix-var}
\mtx{S} \coloneqq \sum_{k=1}^n \mtx{X}_k
\quad\text{and}\quad
v(\mtx{S}) \coloneqq \norm{ \Expect[ \mtx{SS}^* ] } \vee \norm{ \Expect[ \mtx{S}^*\mtx{S} ] }.
\end{equation}
Then the spectral norm of the sum satisfies the probability inequalities
\begin{gather} \label{eqn:bernstein-expect}
\Expect \norm{\mtx{S}} \leq \sqrt{2 v(\mtx{S}) \log(d_1 + d_2)} + \frac{1}{3} B \log(d_1 + d_2); \\
\label{eqn:bernstein-tail}
\Prob{ \norm{\mtx{S}} \geq t} \leq (d_1 + d_2) \exp\left( \frac{-t^2/2}{v(\mtx{S}) + Bt/3} \right)
\quad\text{for $t \geq 0$.}
\end{gather}
\end{theorem}

The proof of \cref{thm:matrix-bernstein} parallels the proof of the scalar Bernstein inequality,
but it requires sophisticated tools from matrix analysis.  The argument will occupy the rest
of this section.  First, we elaborate on the meaning of the result.

The significant outcome of \cref{thm:matrix-bernstein} is the expectation bound~\eqref{eqn:bernstein-expect}.
In contrast, the tail bound~\eqref{eqn:bernstein-tail} is usually somewhat loose;
it can be improved by combining the expectation bound~\eqref{eqn:bernstein-expect}
with scalar concentration inequalities for the norm of an independent
sum, such as~\cite[Prob.~13.33]{BLM13:Concentration-Inequalities}.

We remark that \cref{thm:matrix-bernstein} reproduces the scalar Bernstein inequality
when it is applied to random scalars (i.e., $1 \times 1$ random matrices),
so the numerical constants are sharp.
The dimensional factor $(d_1 + d_2)$ is a new feature in the matrix setting,
and it is a necessary component of the bound; see \cref{sec:bernstein-opt}.

Like the scalar variance, the matrix variance $v(\mtx{S})$ reflects the magnitude of the expected ``square''
of the centered random matrix, where we keep in mind that the non-Hermitian matrix $\mtx{S}$ has two distinct squares
$\mtx{SS}^*$ and $\mtx{S}^*\mtx{S}$.
Both terms in the matrix variance are required, although they coincide when the sum is Hermitian.
It is often convenient to express the matrix variance directly in terms of the summands:
\begin{equation} \label{eqn:matrix-var-sum}
v(\mtx{S}) = \lnorm{ \sum_{k=1}^n \Expect[ \mtx{X}_k \mtx{X}_k^* ] } \vee
	\lnorm{ \sum_{k=1}^n \Expect[ \mtx{X}_k^* \mtx{X}_k ] }.
\end{equation}
The last identity exploits independence and centering.
For simplicity, \cref{thm:matrix-bernstein} assumes that the matrix $\mtx{S}$ is centered;
otherwise, note that
\[
\norm{ \mtx{S} } \leq \norm{ \Expect[ \mtx{S} ] } + \norm{ \mtx{S} - \Expect[\mtx{S}] }.
\]
The theorem now applies to the second term.

\subsection{Reduction to the Hermitian case}

Let us commence with the proof of \cref{thm:matrix-bernstein}.
It suffices to consider the complex case.
Introduce the \term{Hermitian dilation}:
\begin{equation} \label{eqn:hermitian-dilation}
\coll{H} : \C^{d_1 \times d_2} \to \Sym_{d_1 + d_2}(\C)
\quad\text{where}\quad
\coll{H}(\mtx{A}) \coloneqq \begin{bmatrix} \mtx{0} & \mtx{A} \\ \mtx{A}^* & \mtx{0} \end{bmatrix}.
\end{equation}
The map $\coll{H}$ is real-linear, and it preserves spectral data in the sense that
$\lambda_{\max}(\coll{H}(\mtx{A})) = - \lambda_{\min}(\coll{H}(\mtx{A})) = \norm{ \mtx{A} }$.
As a consequence, to prove \cref{thm:matrix-bernstein},
we can pass to the random Hermitian matrix
\[
\coll{H}(\mtx{S}) = \sum_{k=1}^n \coll{H}(\mtx{X}_k)
\quad\text{where $\Expect[ \coll{H}(\mtx{X}_k) ] = \mtx{0}$ and $\norm{\coll{H}(\mtx{X}_k)} \leq B$.}
\]
Through this device, we can simply instate the assumption that the summands are Hermitian matrices.

\subsection{The matrix Laplace transform method}

To extract information about the distribution of a bounded random \emph{Hermitian}
matrix $\mtx{S} \in \Sym_d(\C)$, define the logarithm of the \term{matrix moment generating function}
(briefly, the \term{matrix log-mgf}):
\[
\mtx{\Xi}_{\mtx{S}}(\theta) \coloneqq \log \Expect \econst^{\theta \mtx{S}}
	\in \Sym_d(\C)
\quad\text{for a parameter $\theta \in \R$.}
\]
We apply a scalar function to an Hermitian matrix by applying the function to each eigenvalue
without modifying the associated eigenspace.  %
Ahlswede \& Winter~\cite{AW02:Strong-Converse}
formulated the ideas in this subsection, while Oliveira~\cite{Oli10:Concentration-Adjacency}
and the author~\cite{Tro12:User-Friendly,Tro15:Introduction-Matrix}
crystallized the arguments.

Even in the matrix setting, we can pursue the same strategy that leads to exponential
tail bounds in the scalar setting~\cite[Chap.~2]{BLM13:Concentration-Inequalities}. %
When the parameter $\theta > 0$,
\begin{equation} \label{eqn:matrix-lt-tail}
\begin{aligned}
\Prob{ \lambda_{\max}(\mtx{S}) \geq t }
	&= \Prob{ \smash{\econst^{\theta \lambda_{\max}(\mtx{S})} \geq \econst^{\theta t} } }
	\leq \econst^{-\theta t} \cdot \Expect \econst^{\theta \lambda_{\max}(\mtx{S})} \\
	&= \econst^{-\theta t} \cdot \Expect \lambda_{\max}( \econst^{\theta \mtx{S}} ) 
	\leq \econst^{-\theta t} \cdot \Expect \trace \econst^{\theta \mtx{S}}
	= \econst^{-\theta t} \cdot \trace \exp( \mtx{\Xi}_{\mtx{S}}(\theta) ).
\end{aligned}
\end{equation}
The first inequality is Markov's.  Afterward, invoke the spectral mapping theorem
to draw the eigenvalue map through the exponential, and note that  the trace of a psd matrix
dominates its maximum eigenvalue.  The emergence of the trace exponential of the matrix log-mgf
unlocks powerful tools for trace functions.  

A similar calculation furnishes a bound for the expectation of the maximum eigenvalue.
For each $\theta > 0$,
\begin{equation} \label{eqn:matrix-lt-expect}
\Expect \lambda_{\max}(\mtx{S})
	= \theta^{-1} \log \exp( \Expect \lambda_{\max}(\theta \mtx{S}) )
	\leq \theta^{-1} \log \Expect \econst^{\lambda_{\max}(\theta \mtx{S})}
	\leq \theta^{-1} \log \Expect \trace \econst^{\theta \mtx{S}}
	= \theta^{-1} \log \trace \exp( \mtx{\Xi}_{\mtx{S}}(\theta) ).
\end{equation}
The first inequality is Jensen's.  Once again, the matrix log-mgf controls
the maximum eigenvalue.

\subsection{Subadditivity of the matrix log-mgf}

In the scalar setting, the log-mgf of a sum of independent real random variables
agrees with the sum of the log-mgfs: %
\[
\log \Expect \econst^{\theta \sum_{k=1}^n X_k}
	= \sum_{k=1}^n \log \Expect \econst^{\theta X_k}
	\quad\text{where $(X_1, \dots, X_n) \in \R^n$ is an independent family}.
\]
The latter formula depends on the fact $\econst^{a + b} = \econst^a \econst^b$
for $a, b \in \R$, a property that catastrophically fails for matrices.

Remarkably, there is a substitute.
I established a \emph{subadditivity rule} for the matrix log-mgf~\cite[Lem.~3.4]{Tro12:User-Friendly}:
\begin{equation} \label{eqn:matrix-cgf-subadd}
\trace \exp\big( \mtx{\Xi}_{\mtx{S}}(\theta) \big)
	\leq \trace \exp\left( \sum_{k=1}^n \mtx{\Xi}_{\mtx{X}_k}(\theta) \right)
	\quad\text{where $\mtx{S} \coloneqq \sum_{k=1}^n \mtx{X}_k$.}
\end{equation}
In this expression, $(\mtx{X}_1, \dots, \mtx{X}_n)$ is a statistically independent
family of bounded, Hermitian random matrices of the same dimension.
In the proof of~\eqref{eqn:matrix-cgf-subadd}, the key ingredient 
is a deep concavity theorem of Lieb~\cite[Thm.~6]{Lie73:Convex-Trace}, which is one
of the crown jewels of matrix analysis.  See~\cite[Chap.~8]{Tro15:Introduction-Matrix}
for a detailed proof of Lieb's result. %

\subsection{The matrix Bernstein log-mgf bound}

The subadditivity rule~\eqref{eqn:matrix-cgf-subadd} reduces the bound on the
matrix log-mgf of a sum to individual bounds on the matrix log-mgfs of the summands.
Classic scalar concentration inequalities~\cite[Ch.~2]{BLM13:Concentration-Inequalities} provide inspiration about
fruitful strategies for controlling the matrix log-mgfs~\cite{Tro12:User-Friendly,Tro15:Introduction-Matrix}.
The matrix Bernstein inequality depends on the same type of log-mgf bound
as the scalar Bernstein inequality~\cite[Thm.~2.10]{BLM13:Concentration-Inequalities}.
Indeed, the result~\cite[Lem.~6.1]{Tro15:Introduction-Matrix} states that
the matrix log-mgf of a bounded, centered, random Hermitian  matrix $\mtx{X} \in \Sym_d(\C)$
satisfies the relation
\begin{equation} \label{eqn:bernstein-log-mgf}
\mtx{\Xi}_{\mtx{X}}(\theta) \psdle \frac{\theta^2/2}{1 - B\abs{\theta} / 3} \cdot \Expect[ \mtx{X}^2 ]
\quad\text{where $\Expect[ \mtx{X} ] = \mtx{0}$ and $\norm{\mtx{X}} \leq B$ and $\abs{\theta} < 3/B$.}
\end{equation}
The semidefinite partial order $\mtx{A} \psdle \mtx{H}$ on Hermitian matrices
means that $\mtx{H} - \mtx{A}$ is psd. %
The proof of~\eqref{eqn:bernstein-log-mgf} tracks the scalar argument.
Expand the exponential function in the matrix log-mgf as a Taylor series,
use the centering to remove the term with degree one,
and apply the norm bound to control the terms with degree two and higher.
We also employ the fact that the matrix logarithm respects
the semidefinite partial order~\cite[Prop.~8.4.4]{Tro15:Introduction-Matrix}.

\subsection{Assembly line} %

We are prepared to establish \cref{thm:matrix-bernstein}.
Consider the sum $\mtx{S} \coloneqq \sum_{k=1}^n \mtx{X}_k$
of independent, random Hermitian matrices with dimension $d$ %
that satisfy $\Expect[ \mtx{X}_k ] = \mtx{0}$ and $\norm{\mtx{X}_k} \leq B$.
Sequence the displays~\cref{eqn:matrix-lt-tail,eqn:matrix-cgf-subadd,eqn:bernstein-log-mgf}
to arrive at the tail inequality
\begin{equation} \label{eqn:bernstein-tail-theta}
\begin{aligned}
\Prob{ \lambda_{\max}(\mtx{S}) \geq t }
	&\leq \econst^{-\theta t} \cdot \trace \exp\left( \frac{\theta^2/2}{1-B\abs{\theta}/3} \sum_{k=1}^n \Expect[ \mtx{X}_k^2 ] \right) \\
	&= \econst^{-\theta t} \cdot \trace \exp\left( \frac{\theta^2/2}{1-B\abs{\theta}/3} \Expect[\mtx{S}^2] \right)
	\leq d \cdot \econst^{-\theta t} \cdot \exp\left( \frac{v(\mtx{S}) \theta^2/2}{1-B\abs{\theta}/3} \right).
\end{aligned}
\end{equation}
The parameter is restricted to the interval $\abs{\theta} < B/3$.
The first inequality depends on the fact that the trace exponential respects
the semidefinite partial order~\cite[Ex.~8.1]{Tro15:Introduction-Matrix}.
The identity exploits the centering and independence of the summands.
Afterward, bound the trace exponential by the dimension $d$ times the maximum eigenvalue,
and use spectral mapping to recognize the matrix variance $v(\mtx{S}) = \lambda_{\max}(\Expect[\mtx{S}^2])$.

Finally, in~\eqref{eqn:bernstein-tail-theta}, set the parameter $\theta = t/(v(\mtx{S}) + Bt/3)$
to reach the finished tail inequality
\begin{equation} \label{eqn:matrix-bernstein-eig-tail}
\Prob{ \lambda_{\max}(\mtx{S}) \geq t }
	\leq d \cdot \exp\left( \frac{-t^2/2}{v(\mtx{S}) + Bt/3} \right).
\end{equation}
To establish the analogous tail bound~\eqref{eqn:bernstein-tail} for a general non-Hermitian sum,
apply the Hermitian dilation~\eqref{eqn:hermitian-dilation} and invoke~\eqref{eqn:matrix-bernstein-eig-tail}.
The bound~\eqref{eqn:bernstein-expect} on the \emph{expected} norm of the sum
follows a similar argument starting from~\eqref{eqn:matrix-lt-expect}.

\subsection{Optimality}
\label{sec:bernstein-opt}

As noted, the expectation bound~\eqref{eqn:bernstein-expect} is the most
significant outcome from \cref{thm:matrix-bernstein}.  %
We can strengthen this bound, but not by very much.
Introduce the \term{tail content} statistic:
\[
B_2 \coloneqq \left( \Expect \max\nolimits_k \norm{\mtx{X}_k}^2 \right)^{1/2}.
\]
The tail content satisfies $B_2 \leq B$; the slackness of this inequality depends on the
particular choice of summands.
Using another style of argument based on symmetrization~\cite{Tro16:Expected-Norm},
we can obtain a two-sided bound of the form
\begin{equation} \label{eqn:matrix-rosenthal}
\sqrt{v(\mtx{S})} \ +\ B_2
	\quad\lesssim\quad \Expect \norm{ \mtx{S} }
	\quad\lesssim\quad \sqrt{ v(\mtx{S}) \log(d_1 + d_2) }\ +\ B_2 \log(d_1 + d_2).
\end{equation}
The paper~\cite{Tro16:Expected-Norm} provides examples to show that each term in~\eqref{eqn:matrix-rosenthal}
is necessary, so we cannot improve the bounds %
without a substantial amount of extra information (about the covariance
structure of the random matrix).
\Cref{sec:beyond}  mentions some situations where improvements are possible.

\section{Independent sum model: Applications}
\label{sec:bernstein-appl}

In this section, we present some contemporary applications of the matrix
Bernstein inequality in several areas of computational mathematics.
As a caveat, these stylized applications may not reflect the intricacies
of each problem.  It is sometimes possible to find a simpler
argument by invoking a more suitable matrix concentration inequality,
and we can occasionally obtain sharper analyses using more
complicated tools; see \cref{sec:beyond}.

\subsection{Active subspace methods}

In engineering design, researchers build computer models
of complex physical systems that are governed by several input parameters.
Key tasks include optimization of parameters,
analysis of sensitivity to changes in parameters,
and quantification of uncertainty about the system output.
For example, an aeronautics engineer designs an airfoil by modifying
its geometry to reach a target lift and drag coefficient.
It is challenging to explore the parameter space of a complicated,
nonlinear model, so it is common to seek reduced models.
One basic methodology is to restrict our attention to the most
salient directions in the input space, which we can identify through
a random sampling procedure.
The analysis of this approach depends on matrix concentration inequalities.
This vignette is adapted from Constantine's book~\cite[Chap.~3]{Con15:Active-Subspaces}.
The mathematics are similar with the classic problem of covariance
estimation; see \cref{sec:uniform-cov}.

Let $f : \R^d \to \R$ be an $L$-Lipschitz, differentiable function
that models a complicated system, and assume that we can evaluate
its gradient $\grad f$ at each point in the parameter space---but
the computational cost is significant.
As part of the model, introduce a random vector $\vct{z} \in \R^d$
that describes a distribution over the parameter space;
this distribution is often interpreted as a Bayesian prior.
To capture the variability of the function, introduce the psd \term{sensitivity matrix}
\begin{equation} \label{eqn:sensitivity-matrix}
\mtx{\Sigma} \coloneqq \Expect\big[ (\grad f(\vct{z})) (\grad f(\vct{z}))^* \big] \in \Sym_d^+(\R).
\end{equation}
The quadratic form in $\mtx{\Sigma}$ reflects the sensitivity of the
function to directional perturbations in the input space:
\[
\vct{a}^* \mtx{\Sigma} \vct{a}
	= \Expect\big[ \abs{ \ip{ \grad f(\vct{z}) }{ \vct{a} } }^2 \big]
	\quad\text{for each $\vct{a} \in \R^d$.}
\]
The subspace spanned by the leading eigenvectors of $\mtx{\Sigma}$
is called an \emph{active subspace} of the model, because it
captures the most salient directions in the input space.

We cannot evaluate the sensitivity matrix $\mtx{\Sigma}$ explicitly,
but we can construct a Monte Carlo approximation:
\begin{equation} \label{eqn:empirical-sensitivity}
\widehat{\mtx{\Sigma}}_n \coloneqq \frac{1}{n} \sum_{k=1}^n (\grad f(\vct{z}_k))(\grad f(\vct{z}_k))^*
\quad\text{where $\vct{z}_k \sim \vct{z}$ iid.}
\end{equation}
How many gradient samples suffice to approximate the sensitivity
matrix $\mtx{\Sigma}$?  Can we determine the directions
in which the model varies substantially?

\begin{theorem}[Sampling active subspaces~\protect{\cite[Thm.~3.7]{Con15:Active-Subspaces}}]
Let $f : \R^d \to \R$ be $L$-Lipschitz and differentiable.
As in~\cref{eqn:sensitivity-matrix,eqn:empirical-sensitivity},
consider the sensitivity matrix $\mtx{\Sigma} \in \Sym_d(\R)$
and the empirical approximation $\widehat{\mtx{\Sigma}}_n \in \Sym_d(\R)$
determined by $n$ gradient samples.  The expectation of the relative approximation error satisfies
\[
\Expect\left[ \frac{ \norm{ \widehat{\mtx{\Sigma}}_n - \mtx{\Sigma} }}{\norm{\mtx{\Sigma}}} \right]
	\leq \sqrt{2\beta_n} + \frac{1}{3} \beta_n
	\quad\text{where}\quad
	\beta_n \coloneqq \frac{L^2 \log(2d)}{\norm{\mtx{\Sigma}}} \cdot \frac{1}{n}.
\]
For $\eps \in (0,1)$, the sample complexity $n \geq 4\eps^{-2} L^2 \log(2d) / \norm{\mtx{\Sigma}}$
results in expected relative error at most $\eps$.
\end{theorem}

The ratio $r \coloneqq L^2/\norm{\mtx{\Sigma}}$ reflects the number of energetic dimensions in the model.
Given $n = \mathcal{O}(r \log d)$ gradient samples,
the estimator $\widehat{\mtx{\Sigma}}_n$ for the sensitivity matrix $\mtx{\Sigma}$
allows us to find this active subspace.
To reduce the relative error to a level $\eps$, the estimator requires an additional
factor of $\eps^{-2}$ samples, so the empirical approximation does not attain high accuracy.
This is a familiar weakness of Monte Carlo methods.

\begin{proof}[Proof sketch]
The error in the Monte Carlo approximation is an instance of the independent sum model.
Abbreviate $\vct{y}_k \coloneqq \grad f(\vct{z}_k)$ for each index $k$.  Since the function
$f$ is $L$-Lipschitz, the random vectors satisfy the uniform bound $\norm{ \vct{y}_k } \leq L$.
Introduce the matrix approximation error
\[
\mtx{S} \coloneqq \widehat{\mtx{\Sigma}}_n - \mtx{\Sigma}
	= \sum_{k=1}^n \frac{1}{n} \big(\vct{y}_k \vct{y}_k^* - \Expect[ \vct{y}_k \vct{y}_k^* ]\big)
	\eqqcolon \sum_{k=1}^n \mtx{X}_k.
\]
The summands $(\mtx{X}_k)$ compose an independent family of bounded, centered, random Hermitian matrices
with common dimension $d$.
The statistics appearing in the matrix Bernstein inequality (\cref{thm:matrix-bernstein}) satisfy
\[
B = \sup \norm{ \mtx{X}_k } \leq \frac{L^2}{n}
\quad\text{and}\quad
v(\mtx{S}) = \norm{ \Expect[ \mtx{S}^2 ] }
	= n \cdot \norm{ \Expect[ \mtx{X}_1^2 ] }
	\leq \frac{1}{n} \norm{ \Expect[(\vct{y}_1 \vct{y}_1^*)^2 ] }
	\leq \frac{L^2}{n} \cdot \norm{ \mtx{\Sigma} }.
\]
The stated result follows from~\eqref{eqn:bernstein-expect}.
\end{proof}

\subsection{Stochastic rounding}
\label{sec:stoch-round}

Conventional computer architectures offer floating-point numbers with
16 decimal digits of precision or more.  Driven by contemporary applications,
computer engineers have started to develop new architectures
with far lower precision---sometimes as few as 4 or 8 bits (i.e., 1--2 digits).
At this extreme, large numerical errors can accumulate as we round successive
calculations to the nearest machine-representable number.
One way to mitigate these errors is to design computer systems that
perform \emph{stochastic rounding}.
To analyze linear algebra computations that employ stochastic rounding,
we can exploit matrix concentration tools.
This section offers an idealized treatment of the simplest problem;
see~\cite{CFH+22:Stochastic-Rounding} for more texture. %

A \term{floating-point number system} is a finite set of machine-representable numbers $\set{F} \subset \R$.
Ignoring underflow and overflow, a (deterministic) rounding rule approximates each real number $a \in \R$
by a floating-point number $\texttt{float}(a) \in \set{F}$ that admits the relative-error bound
\begin{equation} \label{eqn:float-rounding}
\abs{ a - \texttt{float}(a) } \leq \texttt{u} \cdot \abs{a}
\quad\text{where $\texttt{u} \in \R$ is the \term{unit-roundoff} parameter.}
\end{equation}
In IEEE double-precision arithmetic, the {unit-roundoff} parameter $\texttt{u} = 2^{-53} \approx 10^{-16}$,
and the rounding rule must satisfy additional requirements
to meet the standard~\cite[Tabs.~1, 2]{CFH+22:Stochastic-Rounding}.
As an alternative, a \emph{stochastic rounding rule} maps each real number $a \in \R$
to a \emph{random} floating-point number $\texttt{sr}(a) \in \set{F}$ that satisfies
\begin{equation} \label{eqn:stoch-rounding}
\Expect[ \texttt{sr}(a) ] = a
\quad\text{and}\quad
\abs{ a - \texttt{sr}(a) } \leq \texttt{u} \cdot \abs{a}.
\end{equation}
In words, stochastic rounding is unbiased, and it commits a relative
error on the order of the unit roundoff.

Suppose that we round each entry of a real-valued matrix stochastically to the nearest
floating-point number.  How much damage will we do?  How does this bound
compare with deterministic rounding?

\begin{theorem}[Stochastic rounding] \label{thm:stoch-round}
Consider a matrix $\mtx{A} %
\in \R^{d_1 \times d_2}$.
For $p \in \{1,2\}$, define the norms
\[
\norm{\mtx{A}}_{\max} \coloneqq \max\nolimits_{ij} \abs{a_{ij}}
\quad\text{and}\quad
\norm{ \mtx{A} }_{{\rm rc}, p} \coloneqq
(\max\nolimits_i \norm{\mtx{A}(i,:)}_p) \vee (\max\nolimits_j \norm{\mtx{A}(:,j)}_p ).
\]
Form a random matrix $\textup{\texttt{sr}}(\mtx{A}) \in \set{F}^{d_1 \times d_2}$
by independently rounding each entry of $\mtx{A}$, honoring the rule~\eqref{eqn:stoch-rounding}. 
Then
\[
\Expect \norm{ \mtx{A} - \textup{\texttt{sr}}(\mtx{A}) }
	\leq \textup{\texttt{u}} \cdot \left[ \sqrt{2 \norm{\mtx{A}}_{{\rm rc}, 2}^2 \log(d_1 + d_2)} + \frac{1}{3} \norm{\mtx{A}}_{\max} \log(d_1 + d_2) \right].
\]
In contrast, the deterministic rule~\eqref{eqn:float-rounding} only guarantees that
$\norm{ \mtx{A} - \textup{\texttt{float}}(\mtx{A}) } \leq  \textup{\texttt{u}} \cdot \norm{ \mtx{A} }_{{\rm rc},1}$.
\end{theorem}

To interpret the result, it is helpful to consider a square matrix with dimension $d$ whose entries
all have similar magnitudes.  In this case, the stochastic rounding bound is governed by the norm
$\norm{ \mtx{A} }_{{\rm rc},2 }$, which is proportional to $\sqrt{d}$.  Meanwhile, the deterministic
rounding bound depends on $\norm{\mtx{A}}_{{\rm rc}, 1}$, which is proportional to the matrix dimension $d$.  This contrast indicates that stochastic
rounding confers a significant benefit.

\begin{proof}[Proof sketch]
The matrix of rounding errors is an instance of the independent sum model:
\[
\mtx{S} \coloneqq \mtx{A} - \texttt{sr}(\mtx{A}) = \sum_{ij} \eps_{ij} a_{ij} \mathbf{E}_{ij}
\quad\text{where $\Expect[\eps_{ij}] = 0$ and $\abs{\eps_{ij}} \leq \texttt{u}$.}
\]
The summands $\mtx{X}_{ij} \coloneqq \eps_{ij} a_{ij} \mathbf{E}_{ij}$
are bounded, centered independent random matrices with dimension $d_1 \times d_2$.
A short calculation confirms that the parameters in the matrix Bernstein inequality (\cref{thm:matrix-bernstein})
satisfy
\[
B = \sup \norm{\mtx{X}_{ij}} \leq \texttt{u} \cdot \norm{\mtx{A}}_{\max} \quad\text{and}\quad
v(\mtx{S}) = \norm{ \Expect[\mtx{SS}^* ] } \vee \norm{ \Expect[\mtx{S}^*\mtx{S}] }
	\leq \texttt{u}^2 \cdot \norm{ \mtx{A} }_{{\rm rc}, 2}^2.
\]
The bound on the expectation of the error follows directly from~\eqref{eqn:bernstein-expect}.
The error bound for deterministic rounding holds because $\norm{ \mtx{M} } \leq \norm{\mtx{M}}_{{\rm rc}, 1}$
for any matrix $\mtx{M}$; see~\cite[Prob.~5.6P21]{HJ13:Matrix-Analysis-2ed}.
\end{proof} 

In fact, the error estimate for stochastic rounding in \cref{thm:stoch-round} admits
a modest refinement:
\[
\Expect \norm{ \mtx{A} - \textup{\texttt{sr}}(\mtx{A}) }
	\leq 4 \texttt{u} \cdot \left[ \norm{\mtx{A}}_{{\rm rc}, 2} + \norm{\mtx{A}}_{\max} \sqrt{1 + \log(d_1 \wedge d_2) } \right]. %
\]
This bound follows from a specialized result~\cite[Thm.~1.4]{BvH25:Extremal-Random}
for random matrices with independent entries, plus a symmetrization argument~\cite[Lem.~6.4.2]{Ver18:High-Dimensional-Probability}.
For most practical use cases, the improvement is insubstantial.

\subsection{Graph sparsification}
\label{sec:graph-sparse}

A combinatorial graph encodes the pairwise relationships among a family of objects.
We may ask whether it is possible to find a simpler graph (with fewer edges)
that preserves structural properties of the original graph,
such as the weights of vertex cuts and the mixing time of a random walk.
Spielman \& Srivastava~\cite{SS11:Graph-Sparsification} showed how to achieve this goal
by randomly sampling edges from the graph.
We can analyze the procedure with matrix concentration.

Consider a connected, weighted, undirected graph $(\set{V}, w)$.
The graph comprises a set $\set{V} \coloneqq \{1, \dots, n\}$ of vertices
and a symmetric function $w : \set{V} \times \set{V} \to \R_+$
that assigns a nonnegative weight to each pair of vertices.
We require that $w_{ii} = 0$ and $w_{ij} = w_{ji}$ for all $i, j \in \set{V}$.
The number of edges $m \coloneqq \# \{ i < j : w_{ij} > 0 \}$.

We can also represent the graph by means of the \term{graph Laplacian} $\mtx{L} \in \Sym_n^+(\R)$,
which is the psd matrix 
\begin{equation} \label{eqn:graph-laplacian}
\mtx{L} \coloneqq \sum_{i < j} w_{ij} \cdot (\mathbf{e}_i - \mathbf{e}_j)(\mathbf{e}_i - \mathbf{e}_j)^*
	\eqqcolon \sum_{i < j} w_{ij} \cdot \mathbf{\Delta}_{ij}.
\end{equation}
The sum involves $m$ nonzero terms.
Since the graph is connected, $\nullsp(\mtx{L}) = \lspan\{\mathbf{1}\}$.
Without further notice, we restrict $\mtx{L}$ to the $(n-1)$-dimensional subspace where it is nonsingular.
Define the \term{effective resistance} of each vertex pair in the graph: 
\[
\varrho_{ij} \coloneqq \trace[ \mtx{L}^{-1/2} \mtx{\Delta}_{ij} \mtx{L}^{-1/2} ]
\quad\text{for each vertex pair $(i, j)$.}
\]
If the graph models an electrical network whose wires have conductances $w_{ij}$,
then the effective resistance $\varrho_{ij}$ is the voltage required to push
one unit of current from vertex $i$ to vertex $j$.
Observe that $\sum_{i < j} w_{ij} \varrho_{ij} = n-1$.
We can approximate all $m$ of the nonzero effective resistances
in $\mathcal{O}(m \log n)$ arithmetic operations~\cite[Thm.~2]{SS11:Graph-Sparsification}.

Our goal is to produce a \emph{sparse} graph %
whose Laplacian matrix $\widehat{\mtx{L}} \in \Sym_n^+(\R)$ is comparable with
the original Laplacian $\mtx{L}$ in the semidefinite partial order:
\begin{equation} \label{eqn:graph-spectral-equiv}
(1 - \eps) \mtx{L} \psdle \widehat{\mtx{L}} \psdle (1+ \eps) \mtx{L} 
\quad\text{for a parameter $\eps \in (0,1)$.}
\end{equation}
The spectral equivalence~\eqref{eqn:graph-spectral-equiv}
ensures that the sparse graph is structurally
similar with the original graph.
We plan to construct a spectral sparsifier $\widehat{\mtx{L}}$ by randomly
sampling edges %
according to a carefully chosen probability distribution.
Introduce the random matrix $\mtx{Y} \in \Sym_d^+(\R)$
with distribution
\[
\Prob{ \mtx{Y} = \frac{n-1}{\varrho_{ij}} \mtx{\Delta}_{ij} }
	= \frac{w_{ij} \varrho_{ij}}{n-1}
	\quad\text{for each $i < j$ where $\varrho_{ij} > 0$.}
\]
In other words, $\mtx{Y}$ is the Laplacian of a weighted edge
between a pair $(i, j)$ of vertices, chosen with probability
proportional to the weight $w_{ij}$ and the effective
resistance $\varrho_{ij}$.
It is easy to confirm that $\Expect[ \mtx{Y} ] = \mtx{L}$,
so the random matrix $\mtx{Y}$ is an unbiased estimator for the
Laplacian $\mtx{L}$ of the original graph.
By averaging $q$ independent copies of $\mtx{Y}$,
we obtain the Laplacian of a random sparse graph with at most
$q$ edges:
\begin{equation} \label{eqn:sparse-graph}
\widehat{\mtx{L}} \coloneqq \frac{1}{q} \sum_{k=1}^q \mtx{Y}_k
\quad\text{where $\mtx{Y}_k \sim \mtx{Y}$ iid.}
\end{equation}
How many edges $q$ suffice to achieve the approximation
guarantee~\eqref{eqn:graph-spectral-equiv}?

\begin{theorem}[Graph sparsification~\protect{\cite[Thm.~1]{SS11:Graph-Sparsification}}]
Let $\mtx{L}$ be the Laplacian~\eqref{eqn:graph-laplacian} of a weighted graph
on $n$ vertices, and let $\widehat{\mtx{L}}$ be the Laplacian~\eqref{eqn:sparse-graph}
of a random sparse graph with at most $q$ edges.  With probability at least $1 - \delta$,
the random matrix $\widehat{\mtx{L}}$ is an $\eps$-approximation of $\mtx{L}$ in the sense~\eqref{eqn:graph-spectral-equiv}
provided that $q \geq 3 \eps^{-2} n \log(2n / \delta)$.
\end{theorem}

The result ensures that we can approximate a graph on $n$ vertices
by a sparse graph with $q = \mathcal{O}(n \log n)$ edges,
regardless of the number $m$ of edges in the original graph.
We pay an extra factor $\eps^{-2}$
in the number of edges $q$ to achieve accuracy $\eps$.
Although the randomized method is simple and efficient, %
there is a more expensive deterministic algorithm
that finds an $\eps$-spectral sparsifier with just
$q = \mathcal{O}(n/\eps^2)$ edges~\cite{BSS14:Twice-Ramanujan}.

\begin{proof}[Proof sketch]
Define the linear map $\mtx{\Phi}(\mtx{A}) \coloneqq \mtx{L}^{-1/2} \mtx{A} \mtx{L}^{-1/2}$
on symmetric matrices that act on $\range(\mtx{L})$.
The spectral equivalence~\eqref{eqn:graph-spectral-equiv} is the same as the requirement that
$
\norm{ \mtx{\Phi}(\widehat{\mtx{L}} - \mtx{L}) } \leq \eps.
$
The matrix inside the norm can be written as an independent sum
of centered random matrices acting on an $(n-1)$-dimensional space:
\[
\mtx{S} \coloneqq \mtx{\Phi}(\widehat{\mtx{L}} - \mtx{L})
	= \sum_{k=1}^q \frac{1}{q} \mtx{\Phi} (\mtx{Y}_k - \Expect[\mtx{Y}_k])
	\eqqcolon \sum_{k=1}^q \mtx{X}_k.
\]
To activate the matrix Bernstein inequality (\cref{thm:matrix-bernstein}), check that
\[
B = \sup \norm{\mtx{X}_k} = \frac{n-1}{q}
\quad\text{and}\quad
v( \mtx{S} )
	\leq \frac{n-1}{q}.
\]
The result now follows from the tail inequality~\eqref{eqn:bernstein-tail}. 
For details, see~\cite[Sec.~5.2]{Tro19:Matrix-Concentration-LN}.
\end{proof}

\subsection{Quantum tomography}

The state of a finite-dimensional quantum system,
such as a register in a quantum computer,
is characterized by a finite-dimensional psd matrix.
We can probe the system by taking measurements,
but the laws of quantum mechanics imply that
each measurement returns a random number.
As a consequence, methods for reconstructing the state
of a quantum system lead to problems involving random matrices.
Matrix concentration tools offer a quick way
to analyze quantum state estimators.
This section summarizes a lecture by
Richard Kueng (see~\cite[Lec.~3]{Tro19:Matrix-Concentration-LN}),
which distills ideas from a paper
of Gu{\c t}a et al.~\cite{GKKT20:Fast-State-Tomography}.
Related ideas animate the theory of classical
shadows, a major recent advance in quantum
information science~\cite{HKP20:Predicting-Many}.

We model a quantum system by means of its \term{density matrix},
which is a complex psd matrix with trace one.
The convex, compact
set $\coll{D}_d \coloneqq \{ \mtx{\rho} \in \Sym_d^+(\C) : \trace[\mtx{\rho}] = 1 \}$
collects the density matrices of dimension $d$.
Next, a \term{quantum measurement} is a family
$\{\mtx{H}_1, \dots, \mtx{H}_m\} \subset \Sym_d^+(\C)$
of psd matrices that partitions the identity:
$\sum_{j=1}^m \mtx{H}_j = \Id$.
When we apply the measurement to a quantum system
with density matrix $\mtx{\rho} \in \coll{D}_d$, two things happen.
First, we sample a discrete random variable $J$
with distribution
\begin{equation} \label{eqn:borns-rule}
\Prob{ J = j \condbar \mtx{\rho} } = \trace[ \mtx{H}_j \mtx{\rho} ]
\quad\text{for $j = 1, \dots, m$.}
\end{equation}
Second, the quantum system ``collapses''.
Therefore, to characterize a quantum system,
we must prepare several copies in the same state
and measure each one to obtain statistical
evidence about the state.

For simplicity, we consider a special type of quantum measurement,
called a \term{complex projective 2-design}.
Here, each of the measurement matrices has rank one:
$\mtx{H}_j = (d/m) \vct{u}_j \vct{u}_j^*$ where $\vct{u}_j \in \C^d$
is a unit-norm vector.
Moreover, these measurements support the reconstruction property
\begin{equation} \label{eqn:complex-2-design}
\frac{1}{m} \sum_{j=1}^m \trace[ \vct{u}_j \vct{u}_j^* \mtx{A} ] \cdot \vct{u}_j \vct{u}_j^* 
	= \frac{1}{(d+1) d} (\mtx{A} + \trace[\mtx{A}] \cdot \Id)
	\quad\text{for each matrix $\mtx{A} \in \Sym_d(\C)$.}
\end{equation}
The vectors $\vct{u}_1, \dots, \vct{u}_m$ composing a complex projective 2-design are
arranged very regularly over the complex sphere.
Examples include systems of ``mutually unbiased''
orthonormal bases and systems of equiangular lines~\cite[Ex.~3.3]{Tro19:Matrix-Concentration-LN}.
This type of measurement system can stably distinguish
any pair of density matrices.

To reconstruct a quantum system with density matrix $\mtx{\rho} \in \coll{D}_d$,
we perform a measurement to sample the random variable $J$.
Then we build a random Hermitian matrix $\mtx{Y} \in \Sym_d(\C)$
according to the rule
\[
\mtx{Y} = (d+1) \vct{u}_J \vct{u}_J^* - \Id.
\]
The random matrix $\mtx{Y}$ serves as an unbiased estimator for the density matrix:
\[
\Expect[ \mtx{Y} ] = \sum_{j=1}^m ((d+1) \vct{u}_j \vct{u}_j^* - \Id) \cdot \Prob{ J = j \condbar \mtx{\rho} }
	= \mtx{\rho}.
\]
This calculation follows from~\eqref{eqn:borns-rule} and~\eqref{eqn:complex-2-design}.
By repeating the measurement on $n$ independent (i.e., unentangled) copies of the
quantum system, each with density matrix $\mtx{\rho}$, we can obtain
an estimator
\begin{equation} \label{eqn:state-avg-estimator}
\mtx{S}_n \coloneqq \frac{1}{n} \sum_{k=1}^n \mtx{Y}_k
\quad\text{where $\mtx{Y}_k \sim \mtx{Y}$ iid.}
\end{equation}
How many independent measurements suffice to approximate
the underlying density matrix?

\begin{theorem}[Quantum tomography]
Let $\mtx{\rho} \in \coll{D}_d$ be a density matrix with dimension $d$.
Using a complex projective 2-design~\eqref{eqn:complex-2-design},
perform a quantum measurement on each of $n$ independent quantum systems with common state $\mtx{\rho}$.
Form the sample-average estimator $\mtx{S}_n$, as in~\eqref{eqn:state-avg-estimator}.
For $\eps \in (0,1)$,
\[
\Prob{ \norm{ \mtx{S}_n - \mtx{\rho} } \geq \eps } \leq
2d \cdot \exp \left( \frac{-3n\eps^2}{14d} \right).
$$
In particular, the sample complexity $n \geq 5 \eps^{-2} d \log(2d/\delta)$
yields a failure probability no greater than $\delta \in (0,1)$.
\end{theorem}

\begin{proof}[Proof sketch]
To apply the matrix Bernstein inequality (\cref{thm:matrix-bernstein})
to the estimator~\eqref{eqn:state-avg-estimator}, simply compute
\[
B = \sup \norm{n^{-1} \mtx{Y}} = \frac{d}{n}
\quad\text{and}\quad
v(\mtx{S}_n) = \frac{1}{n} \norm{ \Expect[ \mtx{Y}^2 ] } = \frac{1}{n}\norm{ (d-1) \mtx{\rho} + d \, \Id}  \leq \frac{2d-1}{n}.
\]
The statement follows quickly from~\eqref{eqn:bernstein-tail}.
\end{proof}

While the sample average estimator~\eqref{eqn:state-avg-estimator}
is unbiased, it generally does not produce a density matrix.
Moreover, the trace norm yields a more interpretable
distance between density matrices than the spectral norm.
To address these concerns, define the \term{projected state estimator}
using the Frobenius norm:
\[
\widehat{\mtx{\rho}}_n \in \arg\min\nolimits_{\mtx{\sigma} \in \coll{D}_d}\ \norm{\mtx{\sigma} - \mtx{S}_n}_{\rm F}.
\]
Lemma 3.9 of~\cite{Tro19:Matrix-Concentration-LN} ensures that the projected state estimator satisfies
the trace-norm bound
\[
\norm{ \widehat{\mtx{\rho}}_n - \mtx{\rho} }_1
	\leq 4 r \norm{ \mtx{S}_n - \mtx{\rho} }
	\quad\text{where $r \coloneqq \rank(\mtx{\rho})$.}
\]
As a consequence, we arrive at an upper bound for the sample complexity of the projected state estimator:
\[
\Prob{ \norm{ \widehat{\mtx{\rho}}_n - \mtx{\rho} }_1 \geq \eps } \leq \delta
\quad\text{when}\quad
n \geq 80 \eps^{-2} r^2 d \log(2d/\delta).
\]
This bound implies that the projected state estimator $\widehat{\mtx{\rho}}_n$
nearly achieves the \emph{optimal} sample complexity that is possible
in this setting~\cite{Haa+17:Sample-Optimal-Tomography}.

\section{Concentration for matrix martingales}

The full majesty of the framework for exponential matrix concentration
comes into view when we extend it to dynamically evolving sequences of matrices.

\subsection{Matrix martingales}

Fix a probability space.  A \emph{filtration} is an increasing sequence of sigma-algebras
$\coll{F}_0 \subseteq \coll{F}_1 \subseteq \coll{F}_2 \subseteq \cdots$ %
inside the master sigma-algebra.
A \term{matrix martingale} is a sequence $(\mtx{S}_0, \mtx{S}_1, \mtx{S}_2, \dots)$ of complex random matrices
with common dimension $d_1 \times d_2$ that satisfies three properties.
For each index $k = 0, 1, 2, \dots$, the sequence is %

\begin{enumerate} \setlength{\itemsep}{0pt}
\item	\textbf{Adapted.}  Each random matrix $\mtx{S}_k$ is measurable with respect to $\coll{F}_k$. %

\item	\textbf{Integrable.}  The expected norm is finite: $\Expect \norm{\mtx{S}_k} < + \infty$. %

\item	\textbf{Status quo.}  The conditional expectation $\Expect[ \mtx{S}_{k+1} \condbar \coll{F}_{k} ] = \mtx{S}_{k}$. %
\end{enumerate}

\noindent
The \term{difference sequence} consists of the matrices $\mtx{X}_k \coloneqq \mtx{S}_k - \mtx{S}_{k-1}$
with $k \geq 1$.  Each member of the difference sequence is conditionally centered: $\Expect[ \mtx{X}_k \condbar \coll{F}_{k-1} ] = \mtx{0}$.
Matrix martingales model a wide range of examples. %

\begin{example}[Adapted sums]
Let $\mtx{S}_0 \coloneqq \mtx{A}$ be a fixed matrix.  For each index $k \geq 1$,
suppose that $\mtx{S}_k \coloneqq \mtx{S}_{k-1} + \mtx{X}_k$
where the increments are conditionally centered: $\Expect[ \mtx{X}_k \condbar \mtx{X}_{0}, \dots, \mtx{X}_{k-1} ] = \mtx{0}$.
Then $(\mtx{S}_0, \mtx{S}_1, \mtx{S}_2, \dots)$ is a matrix martingale with respect to the filtration
$\coll{F}_k \coloneqq \sigma(\mtx{S}_0,\mtx{X}_1,\dots,\mtx{X}_k)$ defined for $k \geq 0$.
\end{example}

\begin{example}[Adapted products]
Let $\mtx{S}_0 \coloneqq \mtx{A}$ be a fixed matrix. %
Suppose that $\mtx{S}_k \coloneqq \mtx{Y}_k \mtx{S}_{k-1}$
where the factors satisfy %
$\Expect[ \mtx{Y}_k \condbar \mtx{S}_{0}, \dots, \mtx{S}_{k-1} ] = \Id$.
Then $(\mtx{S}_0, \mtx{S}_1, \mtx{S}_2, \dots)$ composes a matrix martingale with respect to the filtration
$\coll{F}_k \coloneqq \sigma(\mtx{S}_0,\mtx{Y}_1,\dots,\mtx{Y}_k)$ defined for $k \geq 0$.
\end{example}

\begin{example}[L{\'e}vy--Doob martingale]
Fix a filtration $\coll{F}_0 \subseteq \coll{F}_1 \subseteq \coll{F}_2 \subseteq \dots$.
Let $\mtx{S}$ be a random $d_1 \times d_2$ matrix with $\Expect \norm{\mtx{S}} < + \infty$.
Construct the sequence of conditional expectations:
\[
\mtx{S}_k \coloneqq \Expect[ \mtx{S} \condbar \coll{F}_k ]
\quad\text{for each $k = 0, 1, 2, \dots$.}
\]
Then $(\mtx{S}_0, \mtx{S}_1, \mtx{S}_2, \dots)$ composes a martingale with respect to the filtration.
\end{example}

To avoid technicalities, we only consider finite martingale sequences
$(\mtx{S}_0, \dots, \mtx{S}_n)$.  Furthermore, we require the elements
of the martingale to be uniformly bounded in the sense that
$\sup \norm{ \mtx{S}_k } \leq B_{\infty} < +\infty$
for each index $k = 0, \dots, n$.
The filtration may be suppressed if it is determined by context.

\subsection{Matrix Freedman}

To analyze a matrix martingale, the most valuable theorem is the matrix Freedman
inequality, originally due to Roberto I.~Oliveira~\cite{Oli10:Concentration-Adjacency}.
The author exploited the subadditivity~\eqref{eqn:matrix-cgf-subadd}
of the matrix log-mgf to refine Oliveira's result and to establish some other matrix martingale inequalities~\cite{Tro11:Freedmans-Inequality}.
The version stated here follows from~\cite[Cor.~1.3, Thm.~3.1]{Tro11:Freedmans-Inequality}.

\begin{theorem}[Matrix Freedman] \label{thm:matrix-freedman}
Consider a finite matrix martingale sequence $(\mtx{S}_0, \dots, \mtx{S}_n)$
consisting of $d_1 \times d_2$ matrices, real or complex.
Suppose that the elements of the difference sequence 
$(\mtx{X}_1, \dots, \mtx{X}_n)$ are uniformly bounded:
$\norm{ \mtx{X}_k } \leq B$.
Introduce the \term{conditional quadratic variation} sequence:
\[
V_k \coloneqq
	\lnorm{ \sum_{j=1}^k \Expect\big[ \mtx{X}_j \mtx{X}_j^* \condbar \coll{F}_{j-1} \big] }
	\vee \lnorm{ \sum_{j=1}^k \Expect\big[ \mtx{X}_j^* \mtx{X}_j \condbar \coll{F}_{j-1} \big] }
	\quad\text{for $k = 1, \dots, n$.}
\]
For parameters $v, t \geq 0$, we have the probability inequalities
\begin{align}
\label{eqn:freedman-opt-tail}
\Prob{ \exists k : \norm{\mtx{S}_k - \mtx{S}_0 } \geq \sqrt{2 v t} + Bt/3 \ \text{and} \
	V_k \leq v } &\leq (d_1+d_2) \, \econst^{-t}; \\
\label{eqn:freedman-weak-tail}
\Prob{ \exists k : \norm{\mtx{S}_k - \mtx{S}_0 } \geq t \ \text{and} \
	V_k \leq v } &\leq (d_1+d_2) \exp\left( \frac{-t^2/2}{v + Bt/3} \right).
\end{align}
\end{theorem}

The rest of this section presents a proof of \cref{thm:matrix-freedman} that combines
ideas from my paper~\cite{Tro11:Freedmans-Inequality} and from the survey of Howard et al.~\cite{HRMS20:Time-Uniform-Chernoff}.
First, we dilate on the meaning of the result.

The conditional quadratic variation $V_k$ is the martingale analogue of the matrix variance; cf.~\eqref{eqn:matrix-var-sum}.
At each step $j$ of the process, we compute the expected ``squares'' of the element $\mtx{X}_j$
in the difference sequence, given data $\coll{F}_{j-1}$ about the previous position
of the martingale, and we combine the results. %
Thus, $V_k$ is a random variable that reflects the total observed volatility
of the matrix martingale up to time $k$.
When the difference sequence $(\mtx{X}_1, \dots, \mtx{X}_n)$ is statistically independent,
each $V_k$ reduces to the matrix variance $v(\mtx{S}_k)$.

The probability inequalities~\cref{eqn:freedman-opt-tail,eqn:freedman-weak-tail}
state that the norm of the martingale element $\norm{ \mtx{S}_k }$ is unlikely to be large when the
conditional quadratic variation $V_k$ is bounded above.
The first inequality~\eqref{eqn:freedman-opt-tail} formulates
a clean bound on the tail probability, while the second
inequality provides a direct bound on the level of $\norm{\mtx{S}_k}$.
Let us emphasize that both estimates provide uniform control
on the \emph{entire trajectory} of the matrix martingale.

\subsection{The log-mgf supermartingale}

As in the proof of \cref{thm:matrix-bernstein}, we can and will employ the
Hermitian dilation~\eqref{eqn:hermitian-dilation} to restrict our
attention to Hermitian random matrices.

Consider a bounded Hermitian matrix martingale $(\mtx{S}_0, \dots, \mtx{S}_n)$ taking values in $\Sym_d(\C)$,
with initial condition $\mtx{S}_0 \coloneqq \mtx{0}$ and with difference sequence $(\mtx{X}_1, \dots, \mtx{X}_n)$.
For a parameter $\theta \in \R$, %
define the \emph{conditional} log-mgfs and their partial sums:
\begin{equation} \label{eqn:log-mgf-process}
\mtx{\Xi}_{\mtx{X}_k \condbar \coll{F}_{k-1}}(\theta) \coloneqq \log \Expect\big[ \econst^{\theta \mtx{X}_k} \condbar \coll{F}_{k-1} \big]
\quad\text{and}\quad
\mtx{W}_k(\theta) \coloneqq \sum_{j=1}^k \mtx{\Xi}_{\mtx{X}_j \condbar \coll{F}_{j-1}}(\theta).
\end{equation}
The element $\mtx{W}_k(\theta) \in \Sym_d(\C)$ of the log-mgf process reflects the total volatility of the martingale up to time $k$.

To track the evolution of the matrix martingale, let us pass to a real-valued random sequence:
\begin{equation} \label{eqn:cgf-supermartingale}
M_0(\theta) \coloneqq d
\quad\text{and}\quad
M_k(\theta) \coloneqq \trace \econst^{ \theta \mtx{S}_k - \mtx{W}_k(\theta) } 
\quad\text{where $k = 1, \dots, n$.}
\end{equation}
Applying the subadditivity rule~\eqref{eqn:matrix-cgf-subadd} for the matrix log-mgf conditionally,
we quickly verify that $\Expect[ M_k(\theta) \condbar \coll{F}_{k-1} ] \leq M_{k-1}(\theta)$
for each $k = 1, \dots, n$.
In other words, $(M_k(\theta) : k =0, \dots, n)$ composes a positive supermartingale.
We can study the trajectory of the matrix martingale by bounding
the supermartingale~\cite{Oli10:Concentration-Adjacency,Tro11:Freedmans-Inequality}.
The approach here is adapted from Howard et al.~\cite{HRMS20:Time-Uniform-Chernoff}.

\begin{proposition}[Matrix martingale: Eigenvalue bounds] \label{prop:matrix-martingale}
Consider a finite matrix martingale $(\mtx{S}_0, \dots, \mtx{S}_n)$ taking values in $\Sym_d(\C)$
with log-mgf process $(\mtx{W}_1(\theta), \dots, \mtx{W}_n(\theta))$, as in \eqref{eqn:log-mgf-process}.
For parameters $\theta, \alpha  > 0$,
\[
\Prob{ \exists k : \lambda_{\max}(\mtx{S}_k) \geq \theta^{-1} \big(\alpha + \lambda_{\max}(\mtx{W}_k(\theta))\big) }
	\leq d \cdot \econst^{- \alpha}.
\]
\end{proposition}

\begin{proof}
The log-mgf supermartingale~\eqref{eqn:cgf-supermartingale} admits a lower bound
involving the eigenvalues of the sequences:
\[
\log M_k(\theta) = \log \trace \econst^{ \theta \mtx{S}_k - \mtx{W}_k(\theta) }
	\geq \theta \lambda_{\max}(\mtx{S}_k) - \lambda_{\max}(\mtx{W}_k(\theta))
	\quad\text{for each $\theta > 0$.}
\]
For the inequality, note that the trace dominates the maximum eigenvalue; then
apply spectral mapping %
and Weyl's perturbation theorem~\cite[Cor.~III.2.6]{Bha97:Matrix-Analysis}.
Ville's inequality for positive supermartingales~\cite[Lem.~1]{HRMS20:Time-Uniform-Chernoff} implies
\[
\Prob{ \exists k : \theta \lambda_{\max}(\mtx{S}_k) - \lambda_{\max}(\mtx{W}_k(\theta)) \geq \alpha }
	\leq \Prob{ \exists k : \log M_k(\theta) \geq \alpha }
	\leq \econst^{- \alpha} \cdot \Expect[ M_0(\theta) ] = d \cdot \econst^{- \alpha}.
\]
This is the advertised result.
\end{proof}

\subsection{Proof of Matrix Freedman}

To invoke \cref{prop:matrix-martingale}, we seek a bound for the maximum eigenvalue
of the log-mgf process $(\mtx{W}_k(\theta) : k = 0, \dots, n)$.
In the setting of \cref{thm:matrix-freedman}, each element of the difference sequence
admits a bound on the conditional log-mgf of the Bernstein type~\eqref{eqn:bernstein-log-mgf}: %
\[
\mtx{\Xi}_{\mtx{X}_j \condbar \coll{F}_{j-1}}(\theta)
	\psdle g(\theta) \cdot \Expect\big[ \mtx{X}_j^2 \condbar \coll{F}_{j-1} \big]
	\quad\text{where}\quad
	g(\theta) \coloneqq \frac{\theta^2/2}{1 - B\abs{\theta}/3}
	\quad\text{and}\quad \abs{\theta} < 3/B.
\]
By Weyl's monotonicity theorem~\cite[Cor.~III.2.3]{Bha97:Matrix-Analysis},
\[
\lambda_{\max}(\mtx{W}_k(\theta))
	= \lambda_{\max}\left( \sum_{j=1}^k \mtx{\Xi}_{\mtx{X}_j \condbar \coll{F}_{j-1}}(\theta) \right)
	\leq g(\theta) \cdot \lambda_{\max}\left(\sum_{j=1}^k \Expect\big[ \mtx{X}_j^2 \condbar \coll{F}_{j-1} \big] \right)
	= g(\theta) \cdot V_k.
\]
\cref{prop:matrix-martingale} yields
\begin{align*}
	d \cdot \econst^{-\alpha}
	&\geq \Prob{ \exists k : \lambda_{\max}( \mtx{S}_k ) \geq \theta^{-1}\big(\alpha + \lambda_{\max}(\mtx{W}_k(\theta))\big) }
	\geq \Prob{ \exists k : \lambda_{\max}( \mtx{S}_k ) \geq \theta^{-1}\big(\alpha + g(\theta) V_k \big) } \\
	&\geq \Prob{ \exists k : \lambda_{\max}( \mtx{S}_k ) \geq \theta^{-1}\big(\alpha + g(\theta) v\big) \ \text{and}\ V_k \leq v }.
\end{align*}
In the Hermitian setting, we obtain an analog of the first matrix Freedman tail bound~\eqref{eqn:freedman-opt-tail}
by selecting $\alpha = t$ and $\theta = \sqrt{t} / (\sqrt{v/2} + B\sqrt{t}/3)$.
The second tail bound~\eqref{eqn:freedman-weak-tail} follows from the first tail bound
by inverting the function of $t$ and making a simple bound; see~\cite[Secs.~2.4 and 2.8]{BLM13:Concentration-Inequalities}.
Finally, we extend to a general (non-Hermitian) matrix martingale by applying
these two bounds to the Hermitian dilation~\eqref{eqn:hermitian-dilation}.

\section{Matrix martingales: Applications}
\label{sec:freedman-appl}

In this section, we outline some applications of the matrix
Freedman inequality in computational mathematics.
This technique is particularly valuable for handling
sequences of matrices that evolve adaptively,
and it provides uniform control over the
entire trajectory of the process.

\subsection{Uniform covariance estimation}
\label{sec:uniform-cov}

Consider a centered random vector $\vct{y} \in \R^d$ that
is uniformly bounded: $\Expect[ \vct{y} ] = \vct{0}$
and $\norm{\vct{y}} \leq L$.  Its covariance matrix
takes the form $\mtx{\Sigma} \coloneqq \Expect[ \vct{yy}^* ]$.
Given iid samples $\vct{y}_1, \vct{y}_2, \vct{y}_3, \dots$
from the distribution, we can construct a sequence
of empirical covariance estimates:
\begin{equation} \label{eqn:covar-seq}
\widehat{\mtx{\Sigma}}_n \coloneqq \frac{1}{n} \sum_{k=1}^n \vct{y}_k \vct{y}_k^*
\quad\text{for each $n = 1, 2, 3, \dots$.}
\end{equation}
We would like to obtain a confidence region
for the entire sequence of covariance estimates,
ensuring simultaneous coverage at all times.
The following statement is adapted from Howard et al.~\cite[Sec.~4.3]{HRMS21:Time-Uniform-Nonparametric}.

\begin{theorem}[Uniform covariance estimation~\protect{\cite[Sec.~4.3]{HRMS21:Time-Uniform-Nonparametric}}]
Let $\vct{y} \in \R^d$ be a bounded, centered random vector with $\norm{\vct{y}} \leq L$
and covariance matrix $\mtx{\Sigma}$.
As in~\eqref{eqn:covar-seq},
construct the sequence $(\widehat{\mtx{\Sigma}}_n : n \in \N)$
of empirical covariance estimates.
The following uniform confidence region attains confidence level $1 - \delta \in (0,1)$:
\[
\Prob{ \forall n \in \N : \frac{\norm{ \widehat{\mtx{\Sigma}}_n - \mtx{\Sigma} }}{\norm{\mtx{\Sigma}}}
	\leq \sqrt{8 \beta_n} + \frac{2}{3} \beta_n } \geq  1- \delta
		\quad\text{where}\quad
		\beta_n \coloneqq \frac{L^2}{\norm{\mtx{\Sigma}}} \cdot \frac{\log(2d/\delta) + \log \log(\econst n)}{n}. 
\]
\end{theorem}

The ratio $r \coloneqq L^2 / \norm{\mtx{\Sigma}}$ reflects the number of dimensions
where the random vector has substantial fluctuation.
We need $n = \mathcal{O}(r \log d)$ samples to reliably estimate
the covariance in these directions.
As the number $n$ of samples continues to increase,
the size of confidence region decreases
at a rate $\sqrt{ n^{-1} \log \log n }$.
The log-log factor is required to correct for occasional extreme fluctuations.
More precise asymptotic results follow from the law of the iterated
logarithm in a Banach space~\cite[Thm.~8.2]{LT91:Probability-Banach}.

\begin{proof}[Proof sketch]
For a parameter $N \in \N$, introduce the finite martingale sequence
\[
\mtx{S}_0 \coloneqq \mtx{0} \quad\text{and}\quad
\mtx{S}_n \coloneqq \sum_{k=1}^n (\vct{y}_k \vct{y}_k^* - \mtx{\Sigma}) \eqqcolon \sum_{k=1}^n \mtx{X}_k
\quad\text{for $n = 1, \dots, N$.}
\]
The difference sequence $(\mtx{X}_k)$ consists of iid centered random matrices of dimension $d$.
A short calculation gives
\[
B = \sup \norm{\mtx{X}_k} \leq L^2
\quad\text{and}\quad
V_n = \lnorm{ \sum_{k=1}^n \Expect[ \mtx{X}_k^2 ] }
	\leq n L^2 \cdot \norm{\mtx{\Sigma}}.
\]
The matrix Freedman inequality~\eqref{eqn:freedman-opt-tail}
ensures that
\[
\Prob{ \exists n \leq N : \norm{ \mtx{S}_n } \geq \sqrt{2 n L^2 \norm{\mtx{\Sigma}} t} + L^2 t / 3 }
	\leq 2d \cdot \econst^{-t}.
\]
Dividing each element $\mtx{S}_n$ by the factor $n \cdot \norm{\mtx{\Sigma}}$
and limiting the index $n$ to the range $[N/2, N]$, we find
\[
\Prob{ \exists n \in [N/2, N] : \frac{\norm{ \widehat{\mtx{\Sigma}}_n - \mtx{\Sigma}}}{\norm{\mtx{\Sigma}}}
	\geq \sqrt{\frac{4 L^2 t}{n \norm{\mtx{\Sigma}}}} + \frac{L^2 t}{ 3n \norm{\mtx{\Sigma}}} } %
	\leq 2d \cdot \econst^{-t}.
\]
The key insight is to apply the last inequality on dyadic intervals: $N = 2^{j}$ for $j = 1, 2, 3, \dots$
at the level $t_j = \log(4 j^2 d /\delta) \leq 2 \log(2d \log(\econst n) /\delta)$.
Combine the results with a union bound, then take the complement.
\end{proof}

\subsection{Cholesky decomposition with stochastic rounding}
\label{sec:chol-round}

Most numerical algorithms proceed in stages,
and the numerical errors compound at each step of the process.
Traditional analyses model these errors deterministically,
and error bounds increase linearly with the number of
steps in the algorithm.
More recently, researchers have started to make a close
accounting of error propagation under probabilistic models,
where the errors accumulate more slowly because of cancelations.
Martingale methods offer an elegant technique for tracking
random errors that depend on the past history of the
algorithm.
Connolly \& Higham~\cite{CH23:Probabilistic-Rounding} have employed this
approach to give a probabilistic analysis of the Householder \textsf{QR} algorithm.
In the same spirit, this section offers a stylized analysis of the Cholesky
factorization algorithm under a stochastic rounding model
(cf.~\cref{sec:stoch-round}).

Let $\mtx{A} \in \Sym_d^{++}(\R)$ be a positive-definite matrix.
For clarity of interpretation, we assume that $\diag(\mtx{A}) = \Id$.
Cholesky's method iteratively produces a factorization
$\mtx{A} = \mtx{C}\mtx{C}^*$ where $\mtx{C} \in \M_d(\R)$ is lower triangular.
Define the initial residual $\mtx{R}_0 \coloneqq \mtx{A}$
and the initial triangular factor $\mtx{C}_0 \coloneqq \mtx{0}$. %
At the $k$th step, reduce the residual $\mtx{R}_{k-1}$ by performing
a Schur complement with respect to its $k$th column:
\[
\vct{c}_k \coloneqq \frac{\mtx{R}_{k-1}(:,k)}{\sqrt{\mtx{R}_{k-1}(k,k)}} \quad\text{and}\quad
\mtx{C}_k \coloneqq \mtx{C}_{k-1} + %
\vct{c}_k \mathbf{e}_k^* \quad\text{and}\quad
\mtx{R}_k \coloneqq \mtx{R}_{k-1} - %
\vct{c}_k \vct{c}_k^*.
\]
This procedure zeroes out the $k$th row and column from $\mtx{R}_{k-1}$, while
placing $\vct{c}_k$ in the $k$th column of $\mtx{C}_k$.
The matrix $\mtx{C}_k$ remains lower triangular.
At each step, the residual and the partial triangular factorization compose the original 
matrix: $\mtx{A} = \mtx{R}_k + \mtx{C}_k \mtx{C}_k^*$.
After $d$ steps, $\mtx{R}_d = \mtx{0}$ and $\mtx{C}_d \mtx{C}_d^* = \mtx{A}$.
Each step requires $\mathcal{O}(d^2)$ arithmetic operations, for a total cost of $\mathcal{O}(d^3)$.

In actual practice, we commit floating-point errors when we reduce
the residual---but we can extract a column from the residual without further error.
If we employ stochastic rounding, then the residual update becomes
\[
\mtx{R}_k \coloneqq \texttt{sr}(\mtx{R}_{k-1} - \vct{c}_k \vct{c}_k^*)
	= \mtx{R}_{k-1} - \vct{c}_k \vct{c}_k^* + \mtx{Y}_k.
\]
The stochastic error $\mtx{Y}_k$ is a random symmetric matrix that depends on
the previous residual.
The sigma-algebra $\coll{F}_{k-1} \coloneqq \sigma(\mtx{R}_0, \dots, \mtx{R}_{k-1})$
captures the history of the algorithm through step $k-1$.
As in \cref{sec:stoch-round}, each entry of $\mtx{Y}_k$ 
satisfies the rounding properties in~\eqref{eqn:stoch-rounding}, conditional on $\coll{F}_{k-1}$,
and each entry is rounded independently of the others (modulo symmetry).
We frame the assumptions that
\begin{equation} \label{eqn:chol-sr-err}
\Expect[ \mtx{Y}_k \condbar \coll{F}_{k-1} ] = \mtx{0}
\quad\text{and}\quad
\lnorm{ \Expect[ \mtx{Y}_k^2 \condbar \coll{F}_{k-1} ] } \leq 2 \texttt{u}^2 \cdot \norm{\mtx{A}}
\quad\text{and}\quad
\norm{ \mtx{Y}_k }_{\max} \leq \texttt{u}.
\end{equation}
The third property follows because $\diag(\mtx{A}) = \Id$
and $\diag(\mtx{R}_{k-1})$ only decreases at each step.  
The second property holds if the computed residual $\mtx{R}_{k-1}$ remains psd
and satisfies $\norm{\mtx{R}_{k-1}} \leq 2 \norm{\mtx{A}}$.
Under mild assumptions, these conditions can be justified
with some additional effort.
We also demand that no underflow or overflow occurs.

The computed residuals $\mtx{R}_k$ and the partial triangular factorizations
$\mtx{C}_k \mtx{C}_k^*$ induce a sequence of approximations of the original matrix:
\[
\mtx{S}_k \coloneqq \mtx{R}_k + \mtx{C}_k \mtx{C}_k^* = \mtx{R}_k + \sum_{j=1}^k \vct{c}_j \vct{c}_j^*
\quad\text{for $k = 0, 1,2,\dots, d$.}
\]
Using these approximations, we can calculate the error in the completed triangular factorization
$\mtx{C} \coloneqq \mtx{C}_d$.
\[
\mtx{CC}^* - \mtx{A}
	= \mtx{S}_d - \mtx{S}_0 = \sum_{k=1}^d (\mtx{S}_k - \mtx{S}_{k-1})
	= \sum_{k=1}^d (\mtx{R}_{k} - \mtx{R}_{k-1} + \vct{c}_k \vct{c}_k^*)
	= \sum_{k=1}^d \mtx{Y}_k.
\]
In other words, the sequence of approximations composes a matrix martingale
with difference sequence $(\mtx{Y}_k : k =1, \dots, d)$.  We can use this
martingale to assess the accumulated rounding error in the Cholesky decomposition.

\begin{theorem}[Cholesky decomposition with stochastic rounding]
Let $\mtx{A} \in \Sym_d^{++}(\R)$ be a $d \times d$ positive-definite matrix with $\diag(\mtx{A}) = \Id$. %
As described above, compute the Cholesky decomposition $\mtx{CC}^*$ with a stochastic rounding
procedure that satisfies the assumption~\eqref{eqn:chol-sr-err}.
The resulting factorization admits the error bound
\[
\Prob{ \norm{\mtx{C}\mtx{C}^* - \mtx{A}}
	\geq \textup{\texttt{u}} \cdot \left( 2\sqrt{d \,\norm{\mtx{A}} \, t} + \frac{1}{3} t \right) }
	\leq 2d \cdot \econst^{-t}.
\]
\end{theorem}

For the Cholesky decomposition, the classic deterministic rounding error analysis~\cite[Chap.~10]{Hig02:Accuracy-Stability}
gives a bound for the entrywise error $\norm{ \mtx{CC}^* - \mtx{A}  }_{\max}$ that is proportional to $\texttt{u} \cdot d$. %
For stochastic rounding, the stylized analysis here indicates that the entrywise error
is typically on the order of $\texttt{u} \cdot \sqrt{d \, \norm{\mtx{A}}}$.
The factor $\sqrt{d}$ improvement can be significant.
A full comparison with deterministic rounding falls outside our ambit.

\begin{proof}[Proof sketch]
We need to consider a more fine-grained martingale to exploit the fact
that matrix entries are rounded independently.
Decompose each error matrix as $\mtx{Y}_k = \sum_{i \leq j} \mtx{X}_{ijk}$,
where $\mtx{X}_{ijk}$ describes the rounding error in the $(i,j)$ and $(j, i)$
entries of $\mtx{Y}_k$.  For each $k$, the family $( \mtx{X}_{ijk} : i \leq j )$
is conditionally independent, given $\coll{F}_{k-1}$.
According to~\eqref{eqn:chol-sr-err}, the parameters in the matrix Freedman inequality (\cref{thm:matrix-freedman}) satisfy
\begin{align*}
B &= \sup \norm{ \mtx{X}_{ijk} }
	= \sup \norm{ \mtx{Y}_k }_{\max}
	\leq \texttt{u}; \\ %
V_d &= \lnorm{ \sum_{k=1}^d \sum_{i\leq j} \Expect\big[\mtx{X}_{ijk}^2 \condbar \coll{F}_{k-1} \big] }
	= \lnorm{ \sum_{k=1}^d \Expect\big[ \mtx{Y}_k^2 \condbar \coll{F}_{k-1} \big] }
	\leq 2 \texttt{u}^2 d \cdot \norm{\mtx{A}}. %
\end{align*}
The result follows from the matrix Freedman inequality~\eqref{eqn:freedman-opt-tail}.
\end{proof}

\subsection{Fast Laplacian solvers}
\label{sec:sparse-chol}

As in \cref{sec:graph-sparse}, consider a weighted, connected, undirected graph
on $n$ vertices with Laplacian matrix $\mtx{L} \in \Sym_n^+(\R)$.
A linear system in the Laplacian matrix takes the form
\begin{equation} \label{eqn:poisson-eqn}
\mtx{L} \vct{u} = \vct{f}
\quad\text{where}\quad
\mathbf{1}^*\vct{u} = \mathbf{1}^* \vct{f} = 0.
\end{equation}
This type of linear system arises in a dizzying range of applications, including
numerical discretizations of elliptic PDEs, clustering and partitioning of data,
and network flow problems~\cite{Ten10:Laplacian-Paradigm}.
The basic algorithm for~\eqref{eqn:poisson-eqn} starts by computing
a Cholesky decomposition $\mtx{L} = \mtx{CC}^*$, and it solves the linear
system by two steps of triangular substitution.
In general, this approach requires $\mathcal{O}(n^3)$ operations,
which is often prohibitive.
Can we do better?

Spielman \& Teng~\cite{ST04:Nearly-Linear-Time} gave an affirmative answer
by devising a theoretical Laplacian solver that runs in time linear
in the number $m$ of edges in the graph and polylogarithmic
in the number $n$ of vertices.
After a  train of refinements, Kyng \& Sachdeva~\cite{KS16:Approximate-Gaussian}
designed a \emph{practical} Laplacian solver that runs in time $\mathcal{O}(m \log^3 n)$.
Their algorithm, \texttt{SparseCholesky}, performs an approximate Cholesky decomposition by randomly
sparsifying the Schur complement at each step. %
The analysis relies on matrix martingales.
See~\cite{Tro19:Matrix-Concentration-LN} for another account.

Here is the intuition.  When we apply the Cholesky method to a graph Laplacian,
each Schur complement step eliminates a vertex from the graph by adding a clique to the
remaining vertices of the graph.  The method is expensive because the clique involves
up to $n - k$ vertices and up to $(n - k)^2$ edges at step $k$.
Yet, as we saw in \cref{sec:graph-sparse}, we can dramatically sparsify a graph
by randomly sampling edges in proportion to their effective resistances.
In this spirit, \texttt{SparseCholesky} starts with residual $\mtx{R}_0 \coloneqq \mtx{L}$
and triangular factor $\mtx{C}_0 \coloneqq \mtx{0}$. It performs the steps
\[
\vct{c}_k \coloneqq \frac{\mtx{R}_{k-1}(:, k)}{\sqrt{\mtx{R}_{k-1}(k,k})}
\quad\text{and}\quad
\mtx{C}_k \coloneqq \mtx{C}_{k-1} + \vct{c}_k \mathbf{e}_k^*
\quad\text{and}\quad
\mtx{R}_k \coloneqq \mtx{R}_{k-1} - \texttt{Sparsify}(\vct{c}_k \vct{c}_k^*).
\]
The \texttt{Sparsify} method produces a random, unbiased approximation
to the matrix $\vct{c}_k \vct{c}_k^*$, which contains the Laplacian
of the clique.  After sparsification, the remaining number of edges
equals the number of neighbors of vertex $k$ in the residual graph $\mtx{R}_{k-1}$.
The full algorithm requires several other innovations:
it splits edges into smaller pieces to reduce their weights;
it maintains estimates for the effective resistances
to implement the sparsification step;
and it eliminates vertices in a random order to limit
the average number of neighbors.
As in \cref{sec:chol-round}, this procedure results in a matrix martingale, namely
$\mtx{S}_k \coloneqq \mtx{R}_k + \mtx{C}_k \mtx{C}_k^*$ for $k = 0, 1, 2, \dots d$.
We can analyze this martingale using the matrix Freedman inequality (\cref{thm:matrix-freedman}).
Here is the outcome.

\begin{theorem}[Sparse Cholesky approximation \protect{\cite{KS16:Approximate-Gaussian}}]
Let $\mtx{L} \in \Sym_d^+(\R)$ be a weighted graph Laplacian, as in~\eqref{eqn:graph-laplacian}.
With probability at least $1 - d^{-1}$,
the \textup{\texttt{SparseCholesky}} algorithm produces an approximate Cholesky factorization
$\mtx{CC}^*$ with the spectral guarantee
\[
0.5 \mtx{L} \psdle \mtx{CC}^* \psdle 1.5 \mtx{L}.
\]
The lower-triangular matrix $\mtx{C}$ has $\mathcal{O}(m \log^2 n)$ nonzero entries.
The expected runtime is $\mathcal{O}(m \log^3 n)$ operations.
\end{theorem}

Once we have computed the approximate Cholesky decomposition
$\mtx{L} \approx \mtx{C}\mtx{C}^*$, we can solve the
linear system~\eqref{eqn:poisson-eqn} using preconditioned
conjugate gradient (PCG) with the triangular preconditioner
$\mtx{C}^{-1}$.

\begin{proposition}[Fast Laplacian solvers \protect{\cite{KS16:Approximate-Gaussian}}]
Given the triangular matrix $\mtx{C}$ computed by the \textup{\texttt{SparseCholesky}} algorithm,
the PCG algorithm solves the consistent Laplacian system~\eqref{eqn:poisson-eqn} to relative
error $\eps$ in the Dirichlet energy norm after $\mathcal{O}(m \log^2(n) \log(1/\eps))$
arithmetic operations.
\end{proposition}

For a sparse graph with $m = \mathcal{O}(n)$ edges,
the total cost of solving the Laplacian system~\eqref{eqn:poisson-eqn} to fixed accuracy
is just $\mathcal{O}(n \log^3 n)$ operations.
For a dense graph with $m = \mathcal{O}(n^2)$ edges, the total cost
is $\mathcal{O}(n^2 \log^3 n)$ operations.
Both results compare favorably with the $\coll{O}(n^3)$ cost
of the classic method based on a full Cholesky factorization.

\subsection{Randomized Trotter formulas}

The evolution of a quantum mechanical system is governed by
a Hamiltonian matrix $\mtx{H} \in \Sym_d(\C)$.
The state of the system at time $t \geq 0$ takes the form
\[
\mtx{\Phi}_t(\mtx{\rho})
	\coloneqq \econst^{- \iunit t \mtx{H}} \cdot \mtx{\rho} \cdot \econst^{+ \iunit t \mtx{H}}
	\quad\text{where the initial state $\mtx{\rho} \in \coll{D}_n$.}
\]
A basic goal of quantum information science is to simulate the action $\mtx{\Phi}_t$
induced by a complicated Hamiltonian through the application of simpler Hamiltonians.
Randomized methods can provide effective algorithms for this task.
This section summarizes a result of Chen et al.~\cite{CHKT21:Concentration-Random}
that builds on work of Campbell~\cite{Cam19:Random-Compiler}.

Suppose that the Hamiltonian can be expressed as a sum of many ``simple'' Hermitian terms:
\begin{equation} \label{eqn:hamiltonian-sum}
\mtx{H} \coloneqq \sum_{m=1}^M \mtx{H}_m
\quad\text{with}\quad
L \coloneqq \sum_{m=1}^M \norm{\mtx{H}_m}
\quad\text{and each $\mtx{H}_m \in \Sym_d(\C)$.}
\end{equation}
The statistic $L$ measures the interaction strength within the Hamiltonian.
This model encompasses applications in quantum chemistry and quantum physics.
For clarity, fix the time horizon for the simulation at $t = 1$.
Given a parameter $n \in \N$, construct a random unitary matrix
by drawing a random term from the Hamiltonian according to
an importance sampling distribution :
\begin{equation} \label{eqn:random-unitary-factor}
\mtx{Y} \coloneqq \econst^{-\iunit \mtx{Z} / n}
\quad\text{where}\quad
\Prob{ \mtx{Z} = \frac{L}{\norm{\mtx{H}_j}} \mtx{H}_j } = \frac{ \norm{\mtx{H}_j } }{ L }
\quad\text{for $j = 1, \dots, M$.}
\end{equation}
To approximate the unitary matrix $\mtx{U} \coloneqq \econst^{-\iunit \mtx{H}}$,
form a product of $n$ independent copies of $\mtx{Y}$:
\begin{equation} \label{eqn:random-unitary-product}
\mtx{Q} \coloneqq \mtx{Y}_n \cdots \mtx{Y}_3 \mtx{Y}_2 \mtx{Y}_1
\quad\text{where $\mtx{Y}_k \sim \mtx{Y}$ iid for $k = 1, \dots, n$.}
\end{equation}
In other words, we simulate the action of the full Hamiltonian $\mtx{H}$ by
performing a series of short simulations with the simpler
Hamiltonians $\mtx{H}_m$.
This is a randomized variant of the Lie--Trotter--Suzuki product formulas
that are widely used to approximate the matrix exponential of a sum.
We can analyze the quality of
the simulation using matrix martingale methods.

\begin{theorem}[Random product formulas~\protect{\cite[Thm.~1]{CHKT21:Concentration-Random}}]
Consider a Hamiltonian of the form~\eqref{eqn:hamiltonian-sum} with interaction strength $L$,
and form the unitary matrix $\mtx{U} \coloneqq \econst^{-\iunit \mtx{H}}$.
Fix parameters $\eps, \delta \in (0, 1)$ where $\eps \leq L$.
For $n \geq 40 \eps^{-2} L^2 \log(2d/\delta)$,
construct a random unitary matrix $\mtx{Q}$ with $n$ factors via the random product
formula~\cref{eqn:random-unitary-product}.
Then
\[
\Prob{ \sup\nolimits_{\mtx{\rho} \in \coll{D}_n} \norm{ \mtx{Q} \cdot \mtx{\rho} \cdot \mtx{Q}^* - \mtx{U} \cdot \mtx{\rho} \cdot \mtx{U}^* }_1
	\geq 2 \eps }
	\leq \delta. 
\]
\end{theorem}

This result states that, with high probability, the random product formula~\eqref{eqn:random-unitary-product}
approximates the evolution of the quantum Hamiltonian for \emph{all} initial states.
The number $n$ of factors in the product does not depend on the number $M$ of summands
in the Hamiltonian~\eqref{eqn:hamiltonian-sum}, but it grows quadratically
with the interaction strength $L$.
For comparison, the number of factors in a deterministic product formula depends linearly
on the number $M$ of summands in the Hamiltonian, but more weakly on the interaction strength $L$.
As a consequence, random product formulas are most effective for
short-time simulations of Hamiltonians with many summands.

\begin{proof}[Proof sketch]
The quantity of interest depends
on the spectral-norm distance between the two unitaries:
\[
E \coloneqq \sup\nolimits_{\mtx{\rho} \in \coll{D}_n} \norm{ \mtx{Q} \cdot \mtx{\rho} \cdot \mtx{Q}^* - \mtx{U} \cdot \mtx{\rho} \cdot \mtx{U}^* }_1
	\leq 2 \, \norm{ \mtx{Q} - \mtx{U} }.
\]
This point follows from the triangle inequality,
unitary invariance, and the operator ideal property
of the trace norm. %
The right-hand side consists of a random fluctuation term and a deterministic bias:
\[
\norm{ \mtx{Q} - \mtx{U} } \leq \norm{ \mtx{Q} - \Expect[ \mtx{Q} ] } + \norm{\Expect[\mtx{Q}] - \mtx{U}}
	\eqqcolon \textrm{flux} + \textrm{bias}.
\]
By statistical independence of the factors in~\eqref{eqn:random-unitary-product},
the expectation $\Expect[ \mtx{Q} ] = (\Expect[ \mtx{Y} ])^n$.

The bias term admits a bound that depends on the interaction strength $L$ of the Hamiltonian:
\[
\textrm{bias} = \norm{ (\Expect[\mtx{Y}])^n - (\econst^{-\iunit \mtx{H}/n})^n } %
	\leq n \cdot \norm{ \Expect[ \mtx{Y} ] - \econst^{-\iunit \mtx{H}/n} }
	\leq n \cdot \left( \frac{L}{n} \right)^2 = \frac{L^2}{n}.
\]
The first inequality follows from repeated application of the triangle inequality and unitary invariance,
while the second inequality requires a second-order approximation of the matrix exponential.
See~\cite[Lem.~3.5 and~3.6]{CHKT21:Concentration-Random} for the details.
We insist that $L^2/n \leq \eps/2$ to control the bias term.

For the fluctuation term, we introduce a L{\'e}vy--Doob martingale.
Let $\coll{F}_k \coloneqq \sigma(\mtx{Y}_1, \dots, \mtx{Y}_k)$,
and define
\[
\mtx{S}_k \coloneqq \Expect[ \mtx{Q} \condbar \coll{F}_{k} ]
	= (\Expect \mtx{Y})^{n-k} \mtx{Y}_k \dots \mtx{Y}_1
\quad\text{for $k = 0, \dots, n$.}
\]
Observe that $\mtx{S}_n = \mtx{Q}$ and $\mtx{S}_0 = \Expect[\mtx{Q}]$.
Construct the difference sequence $\mtx{X}_k \coloneqq \mtx{S}_k - \mtx{S}_{k-1}$
for $k =1, \dots, n$.  By first-order approximation of the matrix exponential,
the difference sequence satisfies uniform bounds
\[
\norm{ \mtx{X}_k } \leq \norm{ \mtx{Y}_k - \Expect[ \mtx{Y}_k] } \leq \frac{2L}{n}
\quad\text{and}\quad
\lnorm{ \Expect\big[ \mtx{X}_k \mtx{X}_k^* \condbar \coll{F}_{k-1} \big] } \vee  \lnorm{ \Expect\big[ \mtx{X}_k^* \mtx{X}_k \condbar \coll{F}_{k-1} \big] }
	\leq \norm{ \mtx{Y}_k - \Expect[ \mtx{Y}_k] }^2 \leq \frac{4L^2}{n^2}.
\]
See~\cite[Prop.~3.3]{CHKT21:Concentration-Random} for the details.
The matrix Freedman inequality~\eqref{eqn:freedman-weak-tail} now implies
a probability inequality for $\mathrm{flux} = \norm{\mtx{S}_n - \mtx{S}_0}$
at level $t = \eps/2$.  Combine the bias and fluctuation to get a tail bound for
the quantity $E$.
\end{proof}

\section{Matrix concentration: Extensions}
\label{sec:beyond}

Over the last few years, researchers have made significant advances in
understanding the behavior of the independent sum model.
These results require a detour into the study of Gaussian random matrices,
which serve as ideal models for independent sums.
First, we outline some improvements of matrix concentration for
Gaussian random matrices, and then we explain how Gaussian
models can capture the spectral features of an independent sum
of random matrices.
Last, we mention several other random matrix models
that remain under active study.

\subsection{Variance statistics}

We can describe the behavior of random matrix models more fully
by employing a broader family of variance statistics.
Introduce the real inner product
\[
\ip{ \mtx{B} }{ \mtx{A} }_{\Re} \coloneqq \Re \trace[ \mtx{B}^* \mtx{A} ]
\quad\text{for $\mtx{A}, \mtx{B} \in \F^{d_1 \times d_2}$.}
\]
The variance function of a random matrix $\mtx{S} \in \F^{d_1 \times d_2}$
is defined as
\begin{equation} \label{eqn:variance-fn}
\mathsf{Var}[\mtx{S}] : \F^{d_1 \times d_2} \to \R_+
\quad\text{where}\quad
\mathsf{Var}[ \mtx{S} ](\mtx{A}) \coloneqq \Var[ \ip{ \mtx{S} }{ \mtx{A} }_{\Re} ]
\quad\text{for $\mtx{A} \in \F^{d_1 \times d_2}$.}
\end{equation}
Variance functions are in correspondence with psd quadratic forms on
$\F^{d_1 \times d_2}$, treated as a real linear space.

The variance function $\mathsf{Var}[\mtx{S}]$ packs up all the second-order statistics
of the random matrix.
In particular, it completely determines the matrix variance $v(\mtx{S})$,
defined in~\eqref{eqn:matrix-var}.
We can also introduce the \term{weak variance}, which is related to the tail behavior of the sum:
$
v_{*}(\mtx{S}) \coloneqq \sup\nolimits_{\norm{\mtx{A}}_1 \leq 1} \mathsf{Var}[\mtx{S}](\mtx{A}).
$ %
The weak variance $v_*(\mtx{S})$ is always smaller than the matrix variance $v(\mtx{S})$,
and often much smaller. %

\subsection{Gaussian matrix models}

A random matrix is \term{Gaussian} when each of its linear marginals is Gaussian. %
More precisely, $\mtx{Z} \in \F^{d_1 \times d_2}$ is Gaussian if and only if the %
random variable $\ip{ \mtx{Z} }{ \mtx{A} }_{\Re}$ follows a real Gaussian distribution
for each matrix $\mtx{A} \in \F^{d_1 \times d_2}$.
A Gaussian random matrix is fully characterized by its expectation $\Expect[\mtx{Z}]$
and by its variance function $\mathsf{Var}[\mtx{Z}]$.
For a specified expectation matrix $\mtx{M} \in \F^{d_1 \times d_2}$ and a variance function $\set{V} : \F^{d_1 \times d_2} \to \R_+$,
we denote the (unique) Gaussian distribution with these statistics by $\normal(\mtx{M}, \set{V})$.

\subsection{Second-order matrix Khinchin inequalities}

The norm of a Gaussian matrix $\mtx{Z}$ satisfies an elegant matrix concentration bound,
called the \term{matrix Khinchin inequality}, originating in work of Lust-Piquard~\cite{LP86:Inegalites-Khintchine}:
\begin{equation} \label{eqn:mki}
\sqrt{ (2/\pi) \, v(\mtx{Z}) } \ \leq \ \Expect \norm{\mtx{Z} - \Expect[\mtx{Z}]}
	\ \leq\ \sqrt{2 v(\mtx{Z}) \log(d_1 + d_2).}
\end{equation}
The lower bound involves Gaussian isoperimetry~\cite[Cor.~3]{LO99:Gaussian-Measures}; %
the upper bound relies on exponential matrix concentration arguments~\cite[Thm.~4.1.1]{Tro15:Introduction-Matrix}.
Both inequalities are saturated.

To improve on the matrix Khinchin inequality~\eqref{eqn:mki}, we need extra
information about the variance function of the Gaussian matrix.
Define the \term{interaction energy} to be twice the maximum eigenvalue of the quadratic form: %
\[
w(\mtx{Z}) \coloneqq 2 \sup\nolimits_{\norm{\mtx{A}}_{\rm F} \leq 1} \ \mathsf{Var}[\mtx{Z}](\mtx{A}).
\]
The interaction energy $w(\mtx{Z})$ can be much smaller than the matrix variance $v(\mtx{Z})$,
but they are incomparable. 
Heuristically, the interaction energy is small when $\mtx{Z}$ is ``highly noncommutative.''
This statistic supports a second-order refinement of the matrix Khinchin
inequality~\cite[Cor.~2.2 and Lem.~2.5]{BBV23:Matrix-Concentration}:
\begin{equation} \label{eqn:mki-2}
\Expect \norm{ \mtx{Z} - \Expect[\mtx{Z}] }
	\leq 2 \sqrt{v(\mtx{Z})} + \mathrm{Const} \cdot \left[ v(\mtx{Z}) w(\mtx{Z}) \log^{3}(d_1 + d_2) \right]^{1/4}.
\end{equation}
When $w(\mtx{Z}) \ll v(\mtx{Z})$, then the result~\eqref{eqn:mki-2} improves over~\eqref{eqn:mki}.
As a simple example, consider a real Ginibre matrix $\mtx{G} \in \M_d(\R)$,
which is a $d \times d$ matrix with iid standardized Gaussian entries. The matrix variance $v(\mtx{G}) = d$,
while %
the interaction energy $w(\mtx{G}) = 2$.  Thus,
\begin{align*}
\eqref{eqn:mki} &\quad\text{implies} \quad \Expect \norm{\mtx{G}} \leq \sqrt{2 d \log(2d)}; \\
\quad
\eqref{eqn:mki-2} &\quad\text{implies}\quad \Expect \norm{\mtx{G}} \leq 2 \sqrt{d} + \mathrm{Const} \cdot \big[2d \log^3(2d)\big]^{1/4}.
\end{align*}
The latter inequality is the stronger one, and its first term is numerically sharp.

Roughly speaking, the logarithmic factor in~\eqref{eqn:mki} arises when
the Hermitian dilations of two independent copies $\coll{H}(\mtx{Z}), \coll{H}(\mtx{Z}')$
of the Gaussian matrix ``almost commute'' with each other, in an appropriate sense.
To quantify this insight, inspired by free probability, I introduced a subtle statistic
called the \term{matrix alignment}~\cite[Def.~3.1]{Tro18:Second-Order-Matrix}.
For a special class of Gaussian matrices,
I also established a preliminary version of the second-order Khinchin
inequality~\cite[Thm.~3.1]{Tro18:Second-Order-Matrix}
that reveals how the matrix alignment arises.

The finished result~\eqref{eqn:mki-2} was obtained through additional
insights of Bandeira et al.~\cite{BB21:Spectral-Norm, BBV23:Matrix-Concentration}.
The latter papers use matrix alignment statistics to compare the spectral
norm of a Gaussian matrix with the spectral norm of a free probability
model (namely, an operator semicircle) that admits explicit formulas.
The role of the interaction energy $w(\mtx{Z})$ is to provide a computable bound
for the matrix alignment~\cite[Rem.~3.3]{BB21:Spectral-Norm}.
Similar results hold for other spectral properties, such as the support of the spectrum
and the mean singular value distribution.

\subsection{Independent sums: Gaussian comparison}

Let us return to a more general setting. %
Consider an independent sum of bounded, centered random matrices
with dimension $d_1 \times d_2$:
\[
\mtx{S} \coloneqq %
\sum_{k=1}^n \mtx{X}_k
\quad\text{where $\Expect[ \mtx{X}_k ] = \mtx{0}$ and $\norm{\mtx{X}_k} \leq B$.}
\]
Let $\textsf{V} \coloneqq \textsf{Var}[\mtx{S}]$ be the variance function of the sum. %
The multivariate central limit theorem suggests that we should compare $\mtx{S}$
with the centered Gaussian random matrix $\mtx{Z} \sim \normal(\mtx{0}, \textsf{V})$,
but it is not clear how to quantify the difference between their distributions. %

Under weak assumptions, we can obtain detailed nonasymptotic comparisons for \emph{spectral properties}
of the independent sum $\mtx{S}$ and the Gaussian model $\mtx{Z}$.
Brailovskaya \& van Handel~\cite[Cor.~2.7]{BV24:Universality-Sharp} established that
\begin{equation} \label{eqn:tanya-univ}
\big\lvert \Expect \norm{ \mtx{S} } - \Expect \norm{ \mtx{Z} } \big\rvert
	\ \lesssim\ \left[ v(\mtx{S}) B \log^{2}(d) \right]^{1/3} +  B \log(d)
		+ \sqrt{ v_*(\mtx{S}) \log(d) }
		\quad\text{where $d \coloneqq d_1 + d_2$.}
\end{equation}
In view of the matrix Khinchin inequality~\eqref{eqn:mki},
the Gaussian matrix satisfies $\Expect \norm{\mtx{Z} } \approx \sqrt{v(\mtx{Z})} = \sqrt{v(\mtx{S})}$.
The weak variance term $\sqrt{v_*(\mtx{S})}$ is often negligible.
Therefore, in case $B^2 \ll v(\mtx{S})$, the discrepancy between the expected norms is much smaller
than the expected norm of the Gaussian matrix.
In other words, when the bound $B$ on the summands is sufficiently small, we can accurately approximate
$\Expect \norm{\mtx{S}}$ by way of estimates for $\Expect \norm{\mtx{Z}}$.
The second-order matrix Khinchin inequality~\eqref{eqn:mki-2}
provides one such estimate.

Brailovskaya \& van Handel~\cite{BV24:Universality-Sharp} developed Gaussian comparisons,
similar with~\eqref{eqn:tanya-univ}, for other spectral statistics of the
independent sum, including the support of the spectrum and the mean distribution of the
singular values.  Their arguments rely on Stein's method, cumulant expansions, M{\"o}bius inversion,
and a selection of matrix inequalities.
My paper~\cite{Tro26:Universality-Laws} offers an alternative proof of their results,
based on the method of exchangeable pairs.
I recently obtained another type of Gaussian comparison~\cite{Tro26:Comparison-Theorems-min,Tro26:Comparison-Theorems-max}
for the extreme eigenvalues of an independent sum;
the proof exploits a deep result from matrix analysis,
called Stahl's theorem~\cite{Sta13:Proof-BMV}.

We conclude with a critical observation.
For independent sums, the latest matrix concentration results,
such as~\cref{eqn:mki-2,eqn:tanya-univ}, offer new insights.
Nevertheless, it may be difficult to deploy these inequalities in applications,
such as the ones in \cref{sec:bernstein-appl}, because we often lack %
fine-grained information about the variance function of the random matrix model.
As a consequence, classic matrix concentration tools
(e.g., \cref{thm:matrix-bernstein}) and the more recent
refinements play complementary roles.

\subsection{Other models}

While the matrix concentration theory for the independent sum model is rather complete,
other types of random matrix models remain mysterious and have continued to attract attention.
Citations are limited to a few typical papers, as there is not enough space to do
justice to this rich literature. %

The {nonlinear random matrix model} concerns a matrix-valued function
of underlying random variables that are usually (but not always) independent.
Results for this model include matrix Efron--Stein inequalities~\cite{PMT16:Efron-Stein-Inequalities}
and (local) matrix Poincar{\'e} inequalities~\cite{HT21:Poincare-Inequalities,HT21:Nonlinear-Matrix}.
There are also specialized results for matrix-valued polynomial chaos~\cite{BLN+25:Matrix-Chaos}.
A related challenge arises from random matrix models with dependencies,
such as the sum of fixed matrices modulated by negatively correlated scalar random variables;
some recent progress on these questions appears in~\cite{KKS22:Scalar-Matrix}. %
In addition to matrix martingales, researchers have also treated other types of sequential processes,
such as matrix-valued Markov chains; for example, see \cite{NSW24:Concentration-Inequalities}.
Another line of work~\cite{BGJ+25:Tensor-Concentration} pursues extensions of matrix concentration
to random tensors.

Many of the models in the last paragraph emerged from problems in combinatorics and algorithms,
and the theoretical results have led to satisfying progress. %
Thus, research on matrix concentration tools remains active,
and it continues to make a profound impact on applications.

\section*{Acknowledgments.}

JAT gratefully acknowledges support under ONR Award N00014-24-1-2223
and the Caltech Carver Mead New Adventures Fund.
Many thanks are due to Ethan Epperly for his valuable insights
on a preliminary draft of this article.

\bibliographystyle{siamplain}
\bibliography{tro25-icm-refs}
\end{document}